\documentclass{amsart}

\usepackage{amsmath,amsfonts,amssymb,mathabx,latexsym,xcolor,graphicx,shuffle}
\usepackage{enumitem}

\usepackage{fullpage}

\newtheorem{theorem}{Theorem}[section]
\newtheorem{lemma}[theorem]{Lemma}
\newtheorem{proposition}[theorem]{Proposition}
\newtheorem{corollary}[theorem]{Corollary}
\newtheorem{conjecture}[theorem]{Conjecture}

\theoremstyle{definition}
\newtheorem{definition}[theorem]{Definition}

\theoremstyle{remark}

\numberwithin{equation}{section}

\newcommand{\schubert}{\ensuremath{\mathfrak{S}}}
\newcommand{\key}{\ensuremath{\kappa}}
\newcommand{\lock}{\ensuremath{\mathfrak{L}}}

\newcommand{\fund}{\ensuremath{\mathfrak{F}}}
\newcommand{\mono}{\ensuremath{\mathfrak{M}}}
\newcommand{\kohnert}{\ensuremath{\mathfrak{K}}}
\newcommand{\Kohnert}{\ensuremath{\mathcal{K}}}

\newcommand{\eS}{\ensuremath{\mathcal{E}}}

\newcommand{\D}{\ensuremath{\mathbb{D}}}
\DeclareRobustCommand{\DD}{\ensuremath{\reflectbox{$\mathbb{D}$}}}
\DeclareRobustCommand{\sDD}{\ensuremath{\reflectbox{$_{\mathbb{D}}$}}}

\newcommand{\sort}{\ensuremath{\mathrm{sort}}}

\newcommand{\Up}{\ensuremath{\mathrm{Up}}}

\newcommand{\KD}{\ensuremath{\mathrm{KD}}}
\newcommand{\MKD}{\ensuremath{\mathrm{MKD}}}
\newcommand{\FKD}{\ensuremath{\mathrm{FKD}}}

\newcommand{\LT}{\ensuremath{\mathrm{LT}}}
\newcommand{\QLT}{\ensuremath{\mathrm{QLT}}}

\newcommand{\SSET}{\ensuremath{\mathrm{SSET}}}
\newcommand{\SET}{\ensuremath{\mathrm{SET}}}

\newcommand{\wt}{\ensuremath{\mathrm{wt}}}
\newcommand{\Des}{\ensuremath{\mathrm{Des}}}

\newcommand{\destand}{\ensuremath{\mathrm{dst}}}
\newcommand{\flatten}{\ensuremath{\mathrm{flat}}}

\newlength\cellsize 

\savebox2{%
\begin{picture}(12,12)
\put(0,0){\line(1,0){12}}
\put(0,0){\line(0,1){12}}
\put(12,0){\line(0,1){12}}
\put(0,12){\line(1,0){12}}
\end{picture}}

\newcommand\cellify[1]{\def\thearg{#1}\def\nothing{}%
\ifx\thearg\nothing\vrule width0pt height\cellsize depth0pt%
  \else\hbox to 0pt{\usebox2\hss}\fi%
  \vbox to 12\unitlength{\vss\hbox to 12\unitlength{\hss$#1$\hss}\vss}}

\newcommand\tableau[1]{\vtop{\let\\=\cr
\setlength\baselineskip{-12000pt}
\setlength\lineskiplimit{12000pt}
\setlength\lineskip{0pt}
\halign{&\cellify{##}\cr#1\crcr}}}

\savebox4{% CIRCLE
\begin{picture}(12,12)
\put(6,6){\circle{12}}
\end{picture}}

\newcommand{\cir}[1]{\def\thearg{#1}\def\nothing{}%
\ifx\thearg\nothing\vrule width0pt height\cellsize depth0pt%
  \else\hbox to 0pt{\usebox4\hss}\fi%
  \vbox to 12\unitlength{\vss\hbox to 12\unitlength{\hss$#1$\hss}\vss}}

\newcommand\nocellify[1]{\def\thearg{#1}\def\nothing{}%
\ifx\thearg\nothing\vrule width0pt height\cellsize depth0pt%
  \else\hbox to 0pt{\hss}\fi%
  \vbox to 12\unitlength{\vss\hbox to 12\unitlength{\hss$#1$\hss}\vss}}

\newcommand\notableau[1]{\vtop{\let\\=\cr
\setlength\baselineskip{-12000pt}
\setlength\lineskiplimit{12000pt}
\setlength\lineskip{0pt}
\halign{&\nocellify{##}\cr#1\crcr}}}

\definecolor{boxgray}{gray}{.7}
\newcommand{\cb}{\color{boxgray}\rule{1\cellsize}{1\cellsize}\hspace{-\cellsize}\usebox2}

\usepackage[colorinlistoftodos]{todonotes}

\begin{document}

%%%%%%%%%%%%%%%%%%%%%%%%%%%%%%%%%%%%%%%%%%%%%%%%%%%%%%%%%%%%
%  TITLE PAGE information
%%%%%%%%%%%%%%%%%%%%%%%%%%%%%%%%%%%%%%%%%%%%%%%%%%%%%%%%%%%%

%     [Short Title]{Full Title}
\title[Kohnert polynomials]{Kohnert polynomials}  

\author[S. Assaf]{Sami Assaf}
\address{Department of Mathematics, University of Southern California, 3620 S. Vermont Ave., Los Angeles, CA 90089-2532, U.S.A.}
\email{shassaf@usc.edu}
\thanks{This work was supported by a Collaboration Grant for Mathematicians from the Simons Foundation (Award 524477, S.A.).}

\author[D. Searles]{Dominic Searles}
\address{Department of Mathematics and Statistics, University of Otago, 730 Cumberland St., Dunedin 9016, New Zealand}
\email{dominic.searles@otago.ac.nz}
%\thanks{}

%    General info

%\date{\today}

\dedicatory{To the memory of Axel Kohnert}


\begin{abstract}
  We associate a polynomial to any diagram of unit cells in the first quadrant of the plane using Kohnert's algorithm for moving cells down. In this way, for every weak composition one can choose a cell diagram with corresponding row-counts, with each choice giving rise to a combinatorially-defined basis of polynomials. These \emph{Kohnert bases} provide a simultaneous generalization of Schubert polynomials and Demazure characters for the general linear group. Using the monomial and fundamental slide bases defined earlier by the authors, we show that Kohnert polynomials stabilize to quasisymmetric functions that are nonnegative on the fundamental basis for quasisymmetric functions. For initial applications, we define and study two new Kohnert bases. The elements of one basis are conjecturally Schubert-positive and stabilize to the skew-Schur functions; the elements of the other basis stabilize to a new basis of quasisymmetric functions that contains the Schur functions.
\end{abstract}

\maketitle
\tableofcontents

%%%%%%%%%%%%%%%%%%%%%%%%%%%%%%%%%%%%%%%%%%%%%%%%%%%%%%%%%%%%%%%%
%
\section{Introduction}
%
%%%%%%%%%%%%%%%%%%%%%%%%%%%%%%%%%%%%%%%%%%%%%%%%%%%%%%%%%%%%%%%%
\label{sec:introduction}

Certain homogeneous bases of the ring of polynomials are of central importance in representation theory and geometry. Foremost among these are the Schubert polynomials \cite{LS82}, which are characters of Kra\'skiewicz-Pragacz modules \cite{KP1, KP2} and represent Schubert basis classes in the cohomology of the complete flag variety, and the Demazure characters \cite{Dem74} (also known as key polynomials), which are the characters of Demazure modules for the general linear group. We are motivated by the question of finding other bases of polynomials that exhibit close connections to and share key properties with these important bases. Such bases may be used to understand Schubert polynomials and Demazure characters and moreover may be of independent representation-theoretic or geometric interest.

Kohnert \cite{Koh91} introduced a combinatorial model for the monomial expansion of a Demazure character. This model begins with the diagram $\D(a)$ of a weak composition $a$, the cell diagram in $\mathbb{N}\times \mathbb{N}$ which has $a_i$ cells in row $i$, left-justified. Kohnert defined an algorithmic process on cell diagrams that moves the rightmost cell of a row down to the first available position below. The Kohnert diagrams for $a$ are the cell diagrams that may be obtained by a (possibly empty) sequence of these Kohnert moves on $\D(a)$. Kohnert proved that the Demazure character for $a$ is the generating function of the Kohnert diagrams of $\D(a)$.

Kohnert conjectured that the Schubert polynomials arise by applying the exact same algorithm to different initial cell diagrams, namely, the Rothe diagrams of permutations. Proofs were given by Winkel \cite{Win99,Win02}, though were not fully accepted due to the very technical nature of the arguments; a recent more direct proof was given by Assaf \cite{Ass-R} using the expansion of Schubert polynomials into Demazure characters.

In this work, we study the polynomials arising from application of Kohnert's algorithm to \emph{any} cell diagram in $\mathbb{N}\times \mathbb{N}$; we call these polynomials \emph{Kohnert polynomials}. By definition, Kohnert polynomials expand positively in monomials, and simultaneously generalize both Schubert polynomials and Demazure characters. Given a weak composition $a$, there are several different (though finitely many)  Kohnert polynomials for $a$: in creating an initial cell diagram one must place $a_i$ cells in row $i$, but one may choose in which columns the cells are placed. If one Kohnert polynomial is chosen for every weak composition $a$, we call the resulting set of polynomials a \emph{Kohnert basis} of the polynomial ring. Each Kohnert polynomial in a Kohnert basis has a unique monomial that is minimal in dominance order, hence a Kohnert basis is lower uni-triangular with the basis of monomials. Thus Kohnert bases are bases of the polynomial ring, justifying the nomenclature.

Kohnert bases thus comprise a vast collection of combinatorially-defined bases of polynomials, which includes the Schubert and Demazure character bases. To motivate and facilitate further investigation of Kohnert bases, we prove that \emph{every} Kohnert polynomial expands positively in the monomial slide polynomials introduced in \cite{AS17}. An immediate application is that every Kohnert polynomial has a stable limit, which, in fact, is quasisymmetric and expands positively in the monomial basis of quasisymmetric functions. 

We define necessary and sufficient conditions on cell diagrams for the corresponding Kohnert polynomial to expand positively in the fundamental slide basis \cite{AS17}, a polynomial ring analogue of Gessel's basis of fundamental quasisymmetric functions \cite{Ges84}. While not every Kohnert polynomial expands positively in the fundamental slide basis, we prove that, surprisingly, the stable limit of \emph{any} Kohnert polynomial expands positively in fundamental quasisymmetric functions. For example, the stable limits of Schubert polynomials are Stanley symmetric functions \cite{Mac91} and the stable limits of Demazure characters are Schur polynomials \cite{LS90, AS-2}; each of these is known to expand positively in fundamental quasisymmetric functions. Thus by taking stable limits of Kohnert bases, one obtains new and recovers known families of fundamental-positive quasisymmetric functions. These families may or may not be bases of quasisymmetric functions; for example, the stable limits of Schubert polynomials and Demazure characters are not.

We define a simple condition on diagrams that we conjecture characterizes those diagrams for which the corresponding Kohnert polynomial expands non-negatively as a sum of Demazure characters. Both key diagrams, indexing Demazure characters, and Rothe diagrams, indexing Schubert polynomials, satisfy the stated condition. In further support of the conjecture, the demazure condition is exactly the same as the \emph{northwest} condition of Reiner and Shimozono \cite{RS95-2,RS98} in their study of Specht modules associated to diagrams, suggesting a possible connection between flagged Weyl modules and Kohnert polynomials.

There are several natural choices of ways to associate a two-dimensional cell diagram to a weak composition. Demazure characters arise from left-justification, Schubert polynomials arise from choosing the Rothe diagram of the associated permutation. Using the construction of Kohnert bases, we believe that other natural choices of diagram for a weak composition will yield several new and interesting combinatorial objects, from Kohnert bases of the polynomial ring to families (or bases) of quasisymmetric functions. As a first application of Kohnert bases, we introduce two new bases of polynomials.

The \emph{skew polynomials} are the Kohnert polynomials associated to diagrams arising from certain rightward shifts of contiguous rows of cells. As predicted by our demazure condition, we prove that skew polynomials expand positively into Demazure characters. Based on computer evidence, we conjecture they also expand positively in Schubert polynomials, suggesting a hidden connection with geometry. Stable limits of skew polynomials are symmetric, and in fact are the skew-Schur functions.

The \emph{lock polynomials} are the Kohnert polynomials associated to right-justified diagrams.  Lock polynomials expand positively in fundamental slide polynomials (as do the Schubert polynomials and Demazure characters), and coincide with Demazure characters when the nonzero entries of $a$ are weakly decreasing (which is also the only case when the diagrams satisfy our conjectured demazure condition). The stable limits of lock polynomials, which we call the \emph{extended Schur functions}, are a new basis of quasisymmetric functions. By the theory of Kohnert bases, the extended Schur functions expand positively in the fundamental basis. Per the name, the extended Schur function basis contains the Schur functions as a subset, thus is a lifting of the Schur basis from symmetric to quasisymmetric functions. The description in terms of cell diagrams naturally gives rise to families of tableaux generating the lock polynomials and the extended Schur functions, similar to the definition of Kohnert tableaux for Demazure characters in \cite{AS-2}. The tableau description enables us to give explicit formulas for the expansion of an extended Schur function in terms of fundamental quasisymmetric polynomials, and extract further interesting properties.

We expect these bases and others arising as Kohnert bases may, like the Schubert and Demazure character bases, have deep connections to representation theory and geometry. 

%%%%%%%%%%%%%%%%%%%%%%%%%%%%%%%%%%%%%%%%%%%%%%%%%%%%%%%%%%%%%%%%
\subsection*{Acknowledgments}
%%%%%%%%%%%%%%%%%%%%%%%%%%%%%%%%%%%%%%%%%%%%%%%%%%%%%%%%%%%%%%%%

The authors thank Per Alexandersson, Nantel Bergeron, and Vic Reiner for helpful comments and illuminating discussions.

%%%%%%%%%%%%%%%%%%%%%%%%%%%%%%%%%%%%%%%%%%%%%%%%%%%%%%%%%%%%%%%%
%
\section{Kohnert polynomials}
%
%%%%%%%%%%%%%%%%%%%%%%%%%%%%%%%%%%%%%%%%%%%%%%%%%%%%%%%%%%%%%%%%
\label{sec:kohnert}

In Section~\ref{sec:kohnert-diagrams}, we review Kohnert's algorithm that generates a polynomial from a cell diagram in $\mathbb{N} \times \mathbb{N}$ and use this to define \emph{Kohnert polynomials}. We review the motivating examples of Demazure characters in Section~\ref{sec:kohnert-key} and Schubert polynomials in Section~\ref{sec:kohnert-rothe}, presenting both in the context of Kohnert polynomials.

%%%%%%%%%%%%%%%%%%%%%%%%%%%%%%%%%%%%%%%%%%%%%%%%%%%%%%%%%%%%%%%%
\subsection{Kohnert diagrams}
%%%%%%%%%%%%%%%%%%%%%%%%%%%%%%%%%%%%%%%%%%%%%%%%%%%%%%%%%%%%%%%%
\label{sec:kohnert-diagrams}

A \emph{diagram} is an array of finitely many cells in $\mathbb{N} \times \mathbb{N}$. The weight of a diagram $D$, denoted by $\wt(D)$, is the weak composition whose $i$th part is the number of cells in row $i$. For example, four diagrams with weight $(0,2,1,2)$ are shown in Figure~\ref{fig:diagrams}.

\begin{figure}[ht]
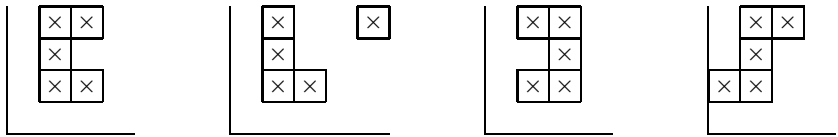

  \begin{center}
    \begin{displaymath}
      \vline\tableau{ & \times & \times \\ & \times \\ & \times & \times \\ & & & \\\hline} \hspace{3\cellsize}
      \vline\tableau{ & \times & & & \times \\ & \times \\ & \times & \times \\ & \\\hline} \hspace{3\cellsize}
      \vline\tableau{ & \times & \times \\ & & \times \\ & \times & \times \\ & & & \\\hline} \hspace{3\cellsize}
      \vline\tableau{ & \times & \times \\ & \times \\ \times & \times \\ & & & \\\hline} 
    \end{displaymath}
    \caption{\label{fig:diagrams}Four diagrams of weight $(0,2,1,2)$.}
  \end{center}
\end{figure}

A diagram is called a \emph{key diagram} if the rows are left justified. For each weak composition $a$, there is a unique key diagram of weight $a$ which we call the \emph{key diagram for $a$} and denote by $\D(a)$. For example, the leftmost diagram in Figure~\ref{fig:diagrams} is the key diagram for $(0,2,1,2)$.

In his thesis, Kohnert \cite{Koh91} described an algorithm for generating a Demazure character, which he called a key polynomial after Lascoux and Sch{\"u}tzenberger \cite{LS90}, from a key diagram by iteratively applying certain \emph{Kohnert moves} to the diagram. 

\begin{definition}[\cite{Koh91}]
  A \emph{Kohnert move} on a diagram selects the rightmost cell of a given row and moves the cell to the first available position below, jumping over other cells in its way as needed. Given a diagram $D$, let $\KD(D)$ denote the set of all diagrams that can be obtained by applying a series of Kohnert moves to $D$. 
\end{definition}

For example, Figure~\ref{fig:KD-key} shows all $16$ Kohnert diagrams for the key diagram $\D(0,2,1,2)$. For comparison, the second diagram in Figure~\ref{fig:diagrams} gives rise to $26$ Kohnert diagrams shown in Figure~\ref{fig:KD-rothe} and the third gives rise to $9$ Kohnert diagrams shown in Figure~\ref{fig:KD-right}.
  
\begin{figure}[ht]
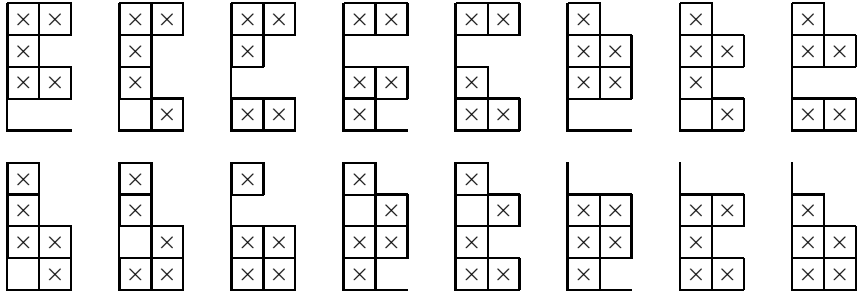

  \begin{center}
    \begin{displaymath}
      \begin{array}{c@{\hskip 1.5\cellsize}c@{\hskip 1.5\cellsize}c@{\hskip 1.5\cellsize}c@{\hskip 1.5\cellsize}c@{\hskip 1.5\cellsize}c@{\hskip 1.5\cellsize}c@{\hskip 1.5\cellsize}c}
        \vline\tableau{ \times & \times \\ \times & \\ \times & \times \\ & \\\hline} &
        \vline\tableau{ \times & \times \\ \times & \\ \times & \\ & \times \\\hline} &
        \vline\tableau{ \times & \times \\ \times & \\ & \\ \times & \times \\\hline} &
        \vline\tableau{ \times & \times \\ & \\ \times & \times \\ \times & \\\hline} &
        \vline\tableau{ \times & \times \\ & \\ \times & \\ \times & \times \\\hline} &
        \vline\tableau{ \times & \\ \times & \times \\ \times & \times \\ & \\\hline} &
        \vline\tableau{ \times & \\ \times & \times \\ \times & \\ & \times \\\hline} &
        \vline\tableau{ \times & \\ \times & \times \\ & \\ \times & \times \\\hline} \\ \\
        \vline\tableau{ \times & \\ \times & \\ \times & \times \\ & \times \\\hline} &
        \vline\tableau{ \times & \\ \times & \\ & \times \\ \times & \times \\\hline} &
        \vline\tableau{ \times & \\ & \\ \times & \times \\ \times & \times \\\hline} &
        \vline\tableau{ \times & \\ & \times \\ \times & \times \\ \times & \\\hline} &
        \vline\tableau{ \times & \\ & \times \\ \times & \\ \times & \times \\\hline} &
        \vline\tableau{ & \\ \times & \times \\ \times & \times \\ \times & \\\hline} &
        \vline\tableau{ & \\ \times & \times \\ \times & \\ \times & \times \\\hline} &
        \vline\tableau{ & \\ \times & \\ \times & \times \\ \times & \times \\\hline} 
      \end{array}
    \end{displaymath}
    \caption{\label{fig:KD-key}Kohnert diagrams for $\D(0,2,1,2)$.}
  \end{center}
\end{figure}

\begin{definition}
  The \emph{Kohnert polynomial indexed by $D$} is 
  \begin{equation}
    \kohnert_{D} = \sum_{T \in \KD(D)} x_1^{\wt(T)_1} \cdots x_n^{\wt(T)_n}.
  \end{equation}
  \label{def:kohnert_poly}
\end{definition}

For example, from Figure~\ref{fig:KD-key}, we see that
\begin{eqnarray*}
  \kohnert_{\D(0,2,1,2)} & = & x_1^2 x_2^2 x_3 + x_1^2 x_2^2 x_4 + x_1^2 x_2 x_3^2 + 2 x_1^2 x_2 x_3 x_4 + x_1^2 x_2 x_4^2 + x_1^2 x_3^2 x_4 + x_1^2 x_3 x_4^2 + x_1 x_2^2 x_3^2 \\
  & & + 2 x_1 x_2^2 x_3 x_4 + x_1 x_2^2 x_4^2 + x_1 x_2 x_3^2 x_4 + x_1 x_2 x_3 x_4^2 + x_2^2 x_3^2 x_4 + x_2^2 x_3 x_4^2.
\end{eqnarray*}

Note that the diagram of a Kohnert polynomial is not necessarily unique. For instance, if two diagrams differ by insertion or deletion of empty columns, then they necessarily give the same Kohnert polynomial. However, as demonstrated by Theorem~\ref{thm:LR} below, this is not sufficient. It is an interesting question to ask for necessary and sufficient conditions for two diagrams to give the same Kohnert polynomial. 

Given weak compositions $a$ and $b$, say that $b$ \emph{dominates} $a$, denoted by $a \leq b$, if $a_1 + \cdots + a_k \leq b_1 + \cdots + b_k$ for all $k$. Since Kohnert polynomials have a unique leading term that is minimal in dominance order, they provide a simple mechanism for constructing interesting bases of the polynomial ring.

\begin{theorem}
  Given any set of diagrams $\{D_a\}$, one for every weak composition, such that $\wt(D_a) = a$, the corresponding Kohnert polynomials $\{\kohnert_{D_a}\}$ form a basis of the polynomial ring.
  \label{thm:kohnert-basis}
\end{theorem}

\begin{proof}
  For any weak composition $a$ and any diagram $D$ such that $\wt(D_a)=a$, the corresponding Kohnert polynomial $\kohnert_{D}$ expands as
  \[ \kohnert_D = x_1^{a_1} \cdots x_n^{a_n} + \sum_{b > a} c_{a,b} x_1^{b_1} \cdots x_n^{b_n} \]
  for some nonnegative integers $c_{a,b}$, where the sum is over weak compositions $b$ that strictly dominate $a$. In particular, any set of Kohnert polynomials of the form $\{\kohnert_{D_a}\}$ where $\wt(D_a) = a$ is lower uni-triangular with respect to monomials, and thus is also a basis.
\end{proof}

As this concept is central to the current study, we introduce the following terminology.

\begin{definition}
  A basis $\{\mathfrak{B}_a\}$ for polynomials is a \emph{Kohnert basis} if each element $\mathfrak{B}_a$ can be realized as a Kohnert polynomial for some diagram $D$ with $\wt(D)=a$.
  \label{def:kohnert-basis}
\end{definition}

Two important examples of Kohnert bases are Demazure characters and Schubert polynomials, discussed below. In addition to proving general positivity results for Kohnert polynomials, we demonstrate the power of this paradigm by giving a new example of a Kohnert basis in Section \ref{sec:right}.

%%%%%%%%%%%%%%%%%%%%%%%%%%%%%%%%%%%%%%%%%%%%%%%%%%%%%%%%%%%%%%%%
\subsection{Demazure characters}
%%%%%%%%%%%%%%%%%%%%%%%%%%%%%%%%%%%%%%%%%%%%%%%%%%%%%%%%%%%%%%%%
\label{sec:kohnert-key}

Kohnert's original motivation for studying key diagrams arose from characters of Demazure modules for the general linear group \cite{Dem74a}, which may be regarded as truncations of irreducible characters \cite{Dem74}. These polynomials were studied combinatorially by Lascoux and Sch{\"u}tzenberger \cite{LS90}, who call them \emph{key polynomials}. For a nice survey of the combinatorial aspects, see \cite{RS95}; for a recent treatment from Kohnert's perspective, see \cite{AS-2}.

The original definition for Demazure characters is in terms of \emph{divided difference operators}, denoted by $\partial_i$, defined on a polynomial $f$ by
\begin{equation}
  \partial_i f = \frac{f - s_i \cdot f}{x_i-x_{i+1}},
  \label{e:del}
\end{equation}
where $s_i$ is the simple transposition interchanging $i$ and $i+1$ and it acts on polynomials by interchanging $x_i$ and $x_{i+1}$. Extending this, we may define a linear operator $\pi_i$ on polynomials by
\begin{equation}
  \pi_i f = \partial_i \left( x_i f \right).
  \label{e:pi}
\end{equation}
Given a permutation $w$, we may define
\begin{eqnarray*}
  \partial_w & = & \partial_{s_1} \cdots \partial_{s_k} \\
  \pi_w & = & \pi_{s_1} \cdots \pi_{s_k}
\end{eqnarray*}
for any expression $s_1 \cdots s_k = w$ with $k$ minimal. It can be shown that both $\partial_w$ and $\pi_w$ are independent of the choice of reduced expression.

\begin{definition}
  Given a weak composition $a$, the \emph{Demazure character} $\key_a$ is 
  \begin{equation}
    \key_a = \pi_{w(a)} x^{\sort(a)},
    \label{e:key}
  \end{equation}
  where $\sort(a)$ is the weakly decreasing rearrangement of $a$ and $w(a)$ is the shortest permutation that sorts $a$.
  \label{def:key}
\end{definition}

For example, for $a = (0,2,1,2)$, we have $\sort(a) = (2,2,1,0)$ and $w(a) = 2431$, and so
\begin{eqnarray*}
  \key_{(0,2,1,2)} & = & \pi_1 \pi_2 \pi_3 \pi_2 \left(x_1^2 x_2^2 x_3\right) \\
  & = & \pi_1 \pi_2 \pi_3 \left( x_1^2 x_2^2 x_3 + x_1^2 x_2 x_3^2 \right) \\
  & = & \pi_1 \pi_2 \left( x_1^2 x_2^2 x_3 + x_1^2 x_2^2 x_4 + x_1^2 x_2 x_3^2 + x_1^2 x_2 x_3 x_4 + x_1^2 x_2 x_4^2 \right) \\
  & = & \pi_1 \left( x_1^2 x_2^2 x_3 + x_1^2 x_2^2 x_4 + x_1^2 x_2 x_3^2 + 2 x_1^2 x_2 x_3 x_4 + x_1^2 x_2 x_4^2 + x_1^2 x_3^2 x_4 + x_1^2 x_3 x_4^2 \right) \\
  & = & x_1^2 x_2^2 x_3 + x_1^2 x_2^2 x_4 + x_1^2 x_2 x_3^2 + 2 x_1^2 x_2 x_3 x_4 + x_1^2 x_2 x_4^2 + x_1^2 x_3^2 x_4 + x_1^2 x_3 x_4^2 + x_1 x_2^2 x_3^2 \\
  & & + 2 x_1 x_2^2 x_3 x_4 + x_1 x_2^2 x_4^2 + x_1 x_2 x_3^2 x_4 + x_1 x_2 x_3 x_4^2 + x_2^2 x_3^2 x_4 + x_2^2 x_3 x_4^2  .
\end{eqnarray*}
Notice that the final computation agrees with $\kohnert_{\D(a)}$ computed earlier. 

\begin{theorem}[\cite{Koh91}]
  The Demazure character $\key_a$ is equal to the Kohnert polynomial $\kohnert_{\D(a)}$, i.e. 
  \begin{equation}
    \key_a = \kohnert_{\D(a)},
    \label{e:kohnert-key}
  \end{equation}
  where $\D(a)$ is the key diagram for the indexing composition $a$.
  \label{thm:kohnert-key}
\end{theorem}

Kohnert's algorithm for key diagrams precisely gives the monomial expansion of a Demazure character. Therefore Kohnert polynomials are a generalization of Demazure characters.

Macdonald \cite{Mac91} noted that when $a$ is weakly increasing of length $n$, we have $\key_a = s_{\mathrm{rev}(a)}(x_1,\ldots,x_n)$, where $s_{\lambda}$ is the \emph{Schur polynomial} that gives the irreducible characters for the general linear group. The Demazure characters are obtained from the irreducible characters by truncating, and so they are, in general, only partially symmetric. However, they are well-defined under stabilization and in the limit converge to the \emph{Schur functions}. This result is implicit in \cite{LS90} and explicit in \cite{AS-2}.

\begin{proposition}\label{prop:keytoSchur}
  For a weak composition $a$, we have
  \begin{equation}
    \lim_{m\rightarrow\infty} \key_{0^m \times a}(x_1,\ldots,x_{n+m}) = s_{\sort(a)}(x_1,x_2,\ldots),
  \end{equation}
  where $0^m \times a$ denotes the weak composition obtained by pre-pending $m$ zeros to $a$.
  \label{prop:key-stable}
\end{proposition}

We will see below that Kohnert polynomials also stabilize, though not, in general, to symmetric functions.

%%%%%%%%%%%%%%%%%%%%%%%%%%%%%%%%%%%%%%%%%%%%%%%%%%%%%%%%%%%%%%%%
\subsection{Schubert polynomials}
%%%%%%%%%%%%%%%%%%%%%%%%%%%%%%%%%%%%%%%%%%%%%%%%%%%%%%%%%%%%%%%%
\label{sec:kohnert-rothe}

Schubert polynomials were introduced by Lascoux and Sch\"utzenberger \cite{LS82} as polynomial representatives for Schubert classes in the cohomology ring of the flag manifold for the general linear group. That is, they are polynomials indexed by permutations whose structure constants precisely correspond to those for the distinguished linear basis of the cohomology ring. They are defined by the divided difference operators, which Fulton \cite{Ful92} showed have deep connections to modern intersection theory. 

\begin{definition}[\cite{LS82}]
  Given a permutation $w$, the \emph{Schubert polynomial} $\schubert_w$ is given by
  \begin{equation}
    \schubert_w = \partial_{w^{-1} w_0} \left( x_1^{n-1} x_2^{n-2} \cdots x_{n-1} \right),
    \label{e:schubert}
  \end{equation}
  where $w_0 = n \cdots 2 1$ is the longest permutation of length $\binom{n}{2}$.
  \label{def:schubert}
\end{definition}

For example, for $w = 143625$, we have $w^{-1}w_0 = 462351$, and so
\begin{eqnarray*}
  \schubert_{143625} & = & \pi_1 \pi_2 \pi_3 \pi_4 \pi_5 \pi_4 \pi_2 \pi_3 \pi_1 \pi_2 \left(x_1^5 x_2^4 x_3^3 x_2^2 x_1 \right) \\
  & = & x_1^3 x_2 x_3 + x_1^3 x_2 x_4 + x_1^3 x_3 x_4 + 2 x_1^2 x_2^2 x_3 + 2 x_1^2 x_2^2 x_4 + x_1^2 x_2 x_3^2 + 3 x_1^2 x_2 x_3 x_4 \\
  & & + x_1^2 x_2 x_4^2 + x_1^2 x_3^2 x_4 + x_1^2 x_3 x_4^2 + x_1 x_2^3 x_3 + x_1 x_2^3 x_4 + x_1 x_2^2 x_3^2 + 3 x_1 x_2^2 x_3 x_4 \\
  & & + x_1 x_2^2 x_4^2 + x_1 x_2 x_3^2 x_4 + x_1 x_2 x_3 x_4^2 + x_2^3 x_3 x_4 + x_2^2 x_3^2 x_4 + x_2^2 x_3 x_4^2 .
\end{eqnarray*}

For a permutation $w$ with a unique descent at position $k$, we have $\schubert_w = s_{\lambda}(x_1,\ldots,x_k)$, where $\lambda$ is the partition given by $\lambda_{k-i+1} = w_i - k$. In particular, Schubert polynomials contain the Schur polynomials as a special case. In certain cases, including this so-called \emph{grassmannian} case, a Schubert polynomial is equal to a Demazure character.

The \emph{Rothe diagram} of a permutation $w$, denoted by $\D(w)$, is given by
\begin{equation}
  \D(w) = \{ (i,w_j) \mid i<j \mbox{ and } w_i > w_j \}.
  \label{e:rothe}
\end{equation}
For example, the middle diagram in Figure~\ref{fig:diagrams} is the Rothe diagram for $143625$. Macdonald \cite{Mac91} used the Rothe diagram of a permutation to characterize precisely when a Schubert polynomial is equal to a Demazure character. Lascoux and Sch\"utzenberger \cite{LS85} first gave such a characterization in terms of pattern avoidance, and they termed permutations $w$ for which $\schubert_w = \key_a$ \emph{vexillary} permutations.

\begin{proposition}[\cite{Mac91}]
  Given a permutation $w$, the following are equivalent
  \begin{enumerate}[label=(\roman*)]
  \item the row support of any two columns of $\D(w)$ are nested sets;
  \item the column support of any two rows of $\D(w)$ are nested sets;
  \item the Schubert polynomial $\schubert_w$ is equal to a key polynomial.
  \end{enumerate}
  When $\schubert_w = \key_a$, we have $a = \wt(\D(w))$.
  \label{prop:vex}
\end{proposition}

Kohnert observed that his algorithm can be used on the Rothe diagram of a vexillary permutation to compute the Schubert polynomial, and he asserted that his rule worked for Schubert polynomials in general. For example, Figure~\ref{fig:KD-rothe} gives the Kohnert diagrams for $\D(143625)$, where we have deleted the empty column on the left since doing so does not affect the Kohnert polynomial. Note that the corresponding Kohnert polynomial is precisely the Schubert polynomial for $143625$.

\begin{figure}[ht]
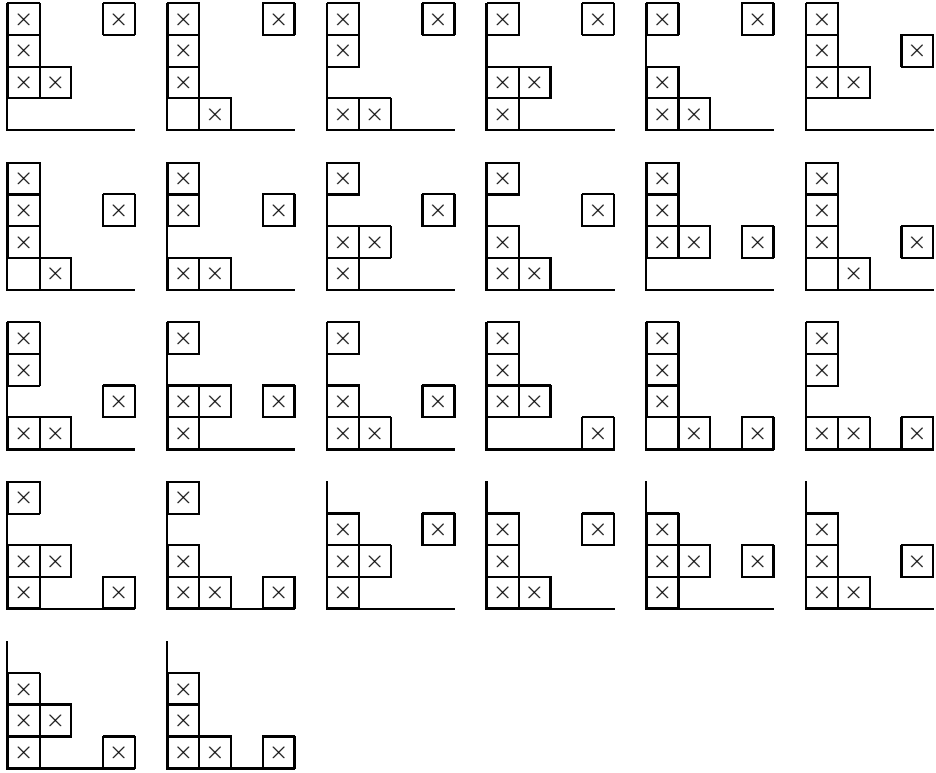

  \begin{center}
    \begin{displaymath}
      \begin{array}{c@{\hskip\cellsize}c@{\hskip\cellsize}c@{\hskip\cellsize}c@{\hskip\cellsize}c@{\hskip\cellsize}c}
      \vline\tableau{ \times & & & \times \\ \times \\ \times & \times \\ \\\hline} & 
      \vline\tableau{ \times & & & \times \\ \times \\ \times & \\ & \times \\\hline} & 
      \vline\tableau{ \times & & & \times \\ \times \\ & \\ \times & \times \\\hline} & 
      \vline\tableau{ \times & & & \times \\ \\ \times & \times \\ \times \\\hline} & 
      \vline\tableau{ \times & & & \times \\ \\ \times & \\ \times & \times \\\hline} & 
      \vline\tableau{ \times \\ \times & & & \times \\ \times & \times \\ \\\hline} \\ \\  
      \vline\tableau{ \times \\ \times & & & \times \\ \times & \\ & \times \\\hline} &
      \vline\tableau{ \times \\ \times & & & \times \\ & \\ \times & \times \\\hline} & 
      \vline\tableau{ \times \\   & & & \times \\ \times & \times \\ \times \\\hline} & 
      \vline\tableau{ \times \\   & & & \times \\ \times & \\ \times & \times \\\hline} & 
      \vline\tableau{ \times \\ \times \\ \times & \times & & \times \\ \\\hline} & 
      \vline\tableau{ \times \\ \times \\ \times &   & & \times \\ & \times \\\hline} \\ \\  
      \vline\tableau{ \times \\ \times \\   &   & & \times \\ \times & \times \\\hline} &
      \vline\tableau{ \times \\   \\ \times & \times & & \times \\ \times \\\hline} &
      \vline\tableau{ \times \\   \\ \times &   & & \times \\ \times & \times \\\hline} & 
      \vline\tableau{ \times \\ \times \\ \times & \times \\   &   & & \times \\\hline} & 
      \vline\tableau{ \times \\ \times \\ \times &   \\   & \times & & \times \\\hline} &
      \vline\tableau{ \times \\ \times \\   &   \\ \times & \times & & \times \\\hline} \\ \\  
      \vline\tableau{ \times \\   \\ \times & \times \\ \times &   & & \times \\\hline} &
      \vline\tableau{ \times \\   \\ \times &   \\ \times & \times & & \times \\\hline} & 
      \vline\tableau{ \\ \times & & & \times \\ \times & \times \\ \times \\\hline} &
      \vline\tableau{ \\ \times & & & \times \\ \times & \\ \times & \times \\\hline} &
      \vline\tableau{ \\ \times & \\ \times & \times & & \times \\ \times \\\hline} & 
      \vline\tableau{ \\ \times & \\ \times &   & & \times \\ \times & \times \\\hline} \\ \\  
      \vline\tableau{ \\ \times & \\ \times & \times \\ \times &   & & \times \\\hline} & 
      \vline\tableau{ \\ \times & \\ \times &   \\ \times & \times & & \times \\\hline} & & & & 
      \end{array}
    \end{displaymath}
    \caption{\label{fig:KD-rothe}Kohnert diagrams for $\D(143625)$.}
  \end{center}
\end{figure}

Two proofs of Kohnert's rule for Schubert polynomials appear in the literature by Winkel \cite{Win99,Win02}, though given the obscure and intricate nature of the arguments, they are not widely accepted. A direct, bijective proof by Assaf \cite{Ass-R} utilizes the expansion of Schubert polynomials into Demazure characters. 

\begin{theorem}[\cite{Win99,Win02,Ass-R}]
  The Schubert polynomial $\schubert_w$ is given by the Kohnert polynomial
  \begin{equation}
    \schubert_w = \kohnert_{\D(w)},
    \label{e:kohnert-rothe}
  \end{equation}
  where $\D(w)$ is the Rothe diagram for the indexing permutation $w$.
  \label{thm:kohnert-rothe}
\end{theorem}

While the Schubert polynomials the contain Schur polynomials, and so are also a polynomial generalization of Schur polynomials, we argue that this fact has more to do with the result that Schubert polynomials expand as nonnegative sums of Demazure characters, and the latter naturally contains Schur polynomials.

Macdonald \cite{Mac91} showed that Schubert polynomials also stabilize and that their stable limits are the Stanley symmetric functions \cite{Sta84} introduced by Stanley to study reduced expressions for a permutation. Stanley \cite{Sta84} proved that these functions are symmetric, and Edelman and Greene \cite{EG87} showed that they are Schur positive. The Schur positivity also follows from Demazure positivity of Schubert polynomials in light of Proposition~\ref{prop:key-stable}.

%%%%%%%%%%%%%%%%%%%%%%%%%%%%%%%%%%%%%%%%%%%%%%%%%%%%%%%%%%%%%%%%
%
\section{Monomial slide expansions}
%
%%%%%%%%%%%%%%%%%%%%%%%%%%%%%%%%%%%%%%%%%%%%%%%%%%%%%%%%%%%%%%%%
\label{sec:mono}

We begin our study of Kohnert polynomials by investigating their expansion in the \emph{monomial slide basis} introduced in \cite{AS17}. In Section~\ref{sec:mono-def}, we review quasisymmetric polynomials and monomial slide polynomials. In Section~\ref{sec:mono-MKD}, we show that every Kohnert polynomial expands nonnegatively into the monomial slide basis allowing us to determine, in particular, when a Kohnert polynomial is quasisymmetric. In Section~\ref{sec:mono-stable}, we use the stable limit of monomial slide polynomials to define \emph{Kohnert quasisymmetric functions}, which are the well-defined stable limits of Kohnert polynomials. 

%%%%%%%%%%%%%%%%%%%%%%%%%%%%%%%%%%%%%%%%%%%%%%%%%%%%%%%%%%%%%%%%
\subsection{Monomial slide polynomials}
%%%%%%%%%%%%%%%%%%%%%%%%%%%%%%%%%%%%%%%%%%%%%%%%%%%%%%%%%%%%%%%%
\label{sec:mono-def}

A polynomial $f \in \mathbb{Z}[x_1,\ldots,x_n]$ is \emph{quasisymmetric} if for every (strong) composition $\alpha = (\alpha_1,\ldots,\alpha_{\ell})$ with $\ell \leq n$, we have
\begin{equation}
  [x_{i_1}^{\alpha_1} \cdots x_{i_{\ell}}^{\alpha_{\ell}} \mid  f] = [x_{j_1}^{\alpha_1} \cdots x_{j_{\ell}}^{\alpha_{\ell}} \mid  f] 
  \label{e:quasisymmetric}
\end{equation}
for any two sequences $1 \leq i_1< \cdots < i_{\ell} \leq n$ and $1 \leq j_1< \cdots < j_{\ell} \leq n$.

The ring of quasisymmetric functions plays a central role in algebraic combinatorics. Gessel \cite{Ges84} initiated the study of quasisymmetric polynomials by introducing the monomial quasisymmetric functions that give an integral basis.

Given a weak composition $a$, let $\flatten(a)$ denote the strong composition obtained by removing all zero parts from $a$. For example, $\flatten(0,2,1,0,2) = (2,1,2)$.

\begin{definition}[\cite{Ges84}]
  The \emph{monomial quasisymmetric polynomial indexed by $\alpha$} is 
  \begin{equation}
    M_{\alpha}(x_1,\ldots,x_n) = \sum_{\flatten(b) = \alpha} x_1^{b_1} \cdots x_n^{b_n} ,
    \label{e:monomial}
  \end{equation}
  where the sum is over all weak compositions of length $n$ whose flattening gives $\alpha$.
  \label{def:monomial}
\end{definition}

For example, take $\alpha = (2,1,2)$ and restricting to $4$ variables, we have
\[ M_{(2,1,2)}(x_1,x_2,x_3,x_4) = x_1^2 x_2 x_3^2 + x_1^2 x_2 x_4^2 + x_1^2 x_3 x_4^2 + x_2^2 x_3 x_4^2 . \]

Assaf and Searles \cite{AS17} introduced a new basis for the polynomial ring, called \emph{monomial slide polynomials}, that gives a natural polynomial generalization of monomial quasisymmetric functions. 

\begin{definition}[\cite{AS17}]
  The \emph{monomial slide polynomial indexed by $a$} is 
  \begin{equation}
    \mono_{a} = \sum_{\substack{b \geq a \\ \flatten(b) = \flatten(a)}} x_1^{b_1} \cdots x_n^{b_n},
    \label{e:monomial-shift}
  \end{equation}
  where $b \geq a$ means $b_1 + \cdots + b_k \geq a_1 + \cdots + a_k$ for all $k=1,\ldots,n$.
  \label{def:monomial-shift}
\end{definition}

For example, taking $a = (2,0,1,2)$, which implies $4$ variables, we compute
\[ \mono_{(2,0,1,2)}(x_1,x_2,x_3,x_4) = x_1^2 x_2 x_3^2 + x_1^2 x_2 x_4^2 + x_1^2 x_3 x_4^2 . \]
As this example illustrates, monomial slide polynomials are not, in general, quasisymmetric. The following result characterizes when a monomial slide polynomial is quasisymmetric.

\begin{proposition}[\cite{AS17}]
  For a weak composition $a$ of length $n$, $\mono_a$ is quasisymmetric in $x_1,\ldots,x_n$ if and only if $a_j\neq 0$ whenever $a_i\neq 0$ for some $i<j$. Moreover, in this case we have $\mono_a = M_{\flatten(a)}(x_1,\ldots,x_n)$.
  \label{prop:lift-mono}
\end{proposition}

The monomial slide polynomials are a lifting of monomial quasisymmetric polynomials to the full polynomial ring. Remarkably, their structure constants are non-negative and generalize the quasi-shuffle product of Hoffman \cite{Hof00}.

\begin{theorem}[\cite{AS17}]
  The monomial slide polynomials $\{\mono_a\}$ are a basis of the polynomial ring with structure constants 
  \begin{equation}
    \mono_{a} \mono_{b} = \sum_{c} [c \mid a \squplus b] \mono_{c},
  \end{equation}
  where $[c \mid a \squplus b]$ is the coefficient of $c$ in the \emph{quasi-slide product} $a \squplus b$. In particular, $[c \mid a \squplus b]$ is a non-negative integer.
\end{theorem}

Unlike Demazure characters and Schubert polynomials, they are not a Kohnert basis.

\begin{proposition}
  Monomial slide polynomials are not a Kohnert basis.
\end{proposition}

\begin{proof}
  For any diagram $D$ of weight $(0,2)$, we claim $\kohnert_D \neq \mono_{(0,2)}$. If $\wt(D)=(0,2)$, then $\KD(D)$ must have a diagram of weight $(1,1)$ by pushing the rightmost box in row $2$. Therefore $\kohnert_D$ will contain the monomial $x_1 x_2$, which does not appear in $\mono_{(0,2)} = x_2^2 + x_1^2$. Therefore $\mono_{(0,2)}$ is not a Kohnert polynomial.
\end{proof}

%%%%%%%%%%%%%%%%%%%%%%%%%%%%%%%%%%%%%%%%%%%%%%%%%%%%%%%%%%%%%%%%
\subsection{Kohnert polynomials are monomial slide positive}
%%%%%%%%%%%%%%%%%%%%%%%%%%%%%%%%%%%%%%%%%%%%%%%%%%%%%%%%%%%%%%%%
\label{sec:mono-MKD}

For every row index $r \geq 1$, define an operator $\Up_r$ on diagrams that raises all cells in row $r$ up to row $r+1$. While this is not, in general, well-defined, we will only apply $\Up_r$ when no cell in row $r$ sits immediately below a cell in row $r+1$. The following definition allows us to relate Kohnert polynomials with monomial slide polynomials.

\begin{definition}
  For a diagram $D$, define the subset $\MKD(D)$ of Kohnert diagrams for $D$ by
  \begin{equation}
    \MKD(D) = \{ T \in \KD(D) \mid \Up_r(T) \not\in \KD(D) \ \forall r \mbox{ such that } (r+1,c)\not\in D \ \forall \ c \}.
  \end{equation}
  \label{def:MKD}
\end{definition}

For example, for $D$ the third diagram in Figure~\ref{fig:diagrams}, Figure~\ref{fig:MKD} shows the set $\MKD(D)$. Notice that
\[ \kohnert_D = \mono_{(0,2,1,2)} + \mono_{(1,1,1,2)} + \mono_{(2,1,1,1)} + \mono_{(1,2,0,2)} + \mono_{(1,2,1,1)}, \]
which corresponds precisely to the weights of the diagrams in $\MKD(D)$. 

\begin{figure}[ht]
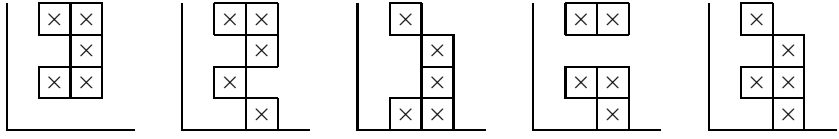

  \begin{center}
    \begin{displaymath}
      \begin{array}{c@{\hskip 1.5\cellsize}c@{\hskip 1.5\cellsize}c@{\hskip 1.5\cellsize}c@{\hskip 1.5\cellsize}c}
        \vline\tableau{ & \times & \times \\  & & \times \\  & \times & \times \\ & & & \\ \hline} &
        \vline\tableau{ & \times & \times \\  & & \times \\  & \times & \\ & & \times & \\ \hline} &
        \vline\tableau{ & \times & \\  & & \times \\  & & \times \\ & \times & \times & \\ \hline} &
        \vline\tableau{ & \times & \times \\  & & \\  & \times & \times \\ & & \times & \\ \hline} &
        \vline\tableau{ & \times & \\  & & \times \\  & \times & \times \\ & & \times & \\ \hline} 
      \end{array}
    \end{displaymath}
    \caption{\label{fig:MKD}The set $\MKD(D)$ for $D$ the leftmost diagram above.}
  \end{center}
\end{figure}

\begin{theorem}
  Given any diagram $D$, we have
  \begin{equation}
    \kohnert_D = \sum_{T \in \MKD(D)} \mono_{\wt(T)} .
  \end{equation}
  In particular, Kohnert polynomials expand non-negatively into monomial slide polynomials.
  \label{thm:kohnert-mono}
\end{theorem}

\begin{proof}
  Given a diagram $T$, we may consider all diagrams $U$ such that applying some sequence of $\Up_i$ moves to $U$ results in $T$. Since $\Up_i$ moves an entire row without consolidating rows, we have $\flatten(U) = \flatten(T)$ for any such $U$. Moreover, since rows move up, we also have $\wt(U) \geq \wt(T)$. Moreover, for any weak composition $b$ such that $\flatten(b)=\flatten(\wt(T))$ and $b \geq \wt(T)$, we may construct $U$ from $b$ and $T$ uniquely by moving the lowest row down to the leftmost nonzero part of $b$, and continuing thus. Therefore we have
  \begin{displaymath}
    \sum_{U \ \mathrm{lifts \ to} \ T} x^{\wt(U)} = \mono_{\wt(T)}.
  \end{displaymath}
  By Definition~\ref{def:MKD}, every Kohnert diagram lifts to some unique element of $\MKD(D)$, independently of the order of the rows that are lifted. Therefore the theorem follows.
\end{proof}

Combining Theorem~\ref{thm:kohnert-mono} and Proposition~\ref{prop:lift-mono}, we have the following characterization of when a Kohnert polynomial is quasisymmetric.

\begin{proposition}
  The polynomial $\kohnert_D$ is quasisymmetric in $x_1,\ldots,x_k$ if and only if for each $i<k$, the set of columns containing cells of $D$ in row $i$ is a subset of the set of columns containing cells of $D$ in row $i+1$.
\end{proposition}

\begin{proof}
  Suppose there is a cell $x$ in row $i<k$ which has no cell immediately above it. Then from right to left, perform a single Kohnert move on all cells in row $i+1$. Let $M$ be the associated monomial of the resulting Kohnert diagram. Then the monomial obtained from $M$ by replacing $x_j$ with $x_{j+1}$ for all $j\le i$ does not belong to $\Kohnert_D$, hence $\Kohnert_D$ is not quasisymmetric.
  
  Conversely, if the rows of $D$ are nested with smaller rows below larger ones, then each element of $\MKD(D)$ will have at least one cell in each nonempty row of $D$. Therefore, by Proposition~\ref{prop:lift-mono}, each terms in the monomial slide polynomial expansion will be quasisymmetric.
\end{proof}

%%%%%%%%%%%%%%%%%%%%%%%%%%%%%%%%%%%%%%%%%%%%%%%%%%%%%%%%%%%%%%%%
\subsection{Kohnert quasisymmetric functions}
%%%%%%%%%%%%%%%%%%%%%%%%%%%%%%%%%%%%%%%%%%%%%%%%%%%%%%%%%%%%%%%%
\label{sec:mono-stable}

The ring of quasisymmetric functions is the inverse limit of quasisymmetric polynomials. The monomial and fundamental quasisymmetric polynomials stabilize to the monomial and fundamental quasisymmetric functions when the number of variables tends to infinity.

Assaf and Searles \cite{AS17} showed that the monomial slide polynomials stabilize and that their stable limits are precisely the monomial quasisymmetric functions.

\begin{theorem}[\cite{AS17}]
  For a weak composition $a$, we have
  \begin{equation}
    \lim_{m\rightarrow\infty}\mono_{0^m \times a} = M_{\flatten(a)}(x_1,x_2,\ldots),
  \end{equation}
  where $0^m \times a$ denotes the weak composition obtained by prepending $m$ $0$'s to $a$.
  \label{thm:mono-stable}
\end{theorem}

Let $0^m \times D$ denote the diagram of $D$ shifted up vertically by $m$ rows. For example, once again taking $D$ to be the third diagram in Figure~\ref{fig:diagrams}, we may compute
\begin{eqnarray*}
  \kohnert_{0\times D} & = & \mono_{(0,0,2,1,2)} + \mono_{(0,1,1,1,2)} + \mono_{(1,1,1,0,2)} + \mono_{(0,2,1,1,1)} \\
  & & + \mono_{(0,1,2,0,2)} + \mono_{(0,1,2,1,1)} + \mono_{(1,1,2,0,1)} + 2\mono_{(1,1,1,1,1)}.
\end{eqnarray*}
Moreover, for any $m \geq 0$, we have the following expansion,
\begin{eqnarray*}
  \kohnert_{0^{m+2}\times D} & = & \mono_{0^{m}\times(0,0,0,2,1,2)} + \mono_{0^{m}\times(0,0,1,1,1,2)} + \mono_{0^{m}\times(0,1,1,1,0,2)} \\
  & & + \mono_{0^{m}\times(0,0,2,1,1,1)} + \mono_{0^{m}\times(0,0,1,2,0,2)} + \mono_{0^{m}\times(0,0,1,2,1,1)} \\
  & & + \mono_{0^{m}\times(0,1,1,2,0,1)} + 2\mono_{0^{m}\times(0,1,1,1,1,1)} + \mono_{0^{m}\times(1,1,1,1,0,1)}.
\end{eqnarray*}
In particular, the monomial slide expansion of the Kohnert polynomial eventually stabilizes. Inspired by this, we may consider the stable limit of Kohnert polynomials in the following sense.

\begin{definition}
  The \emph{Kohnert quasisymmetric function indexed by $D$} is 
  \begin{equation}
    \Kohnert_D(X) = \lim_{m \rightarrow\infty} \kohnert_{0^m \times D},
  \end{equation}
  where $0^m \times D$ denotes the diagram of $D$ shifted up vertically by $m$ rows.
\end{definition}

For example, continuing with $D$ the third diagram in Figure~\ref{fig:diagrams}, we have
\begin{eqnarray*}
  \Kohnert_D & = & M_{(2,1,2)} + 2 M_{(1,1,1,2)} + M_{(2,1,1,1)} + M_{(1,2,2)} + M_{(1,2,1,1)} + M_{(1,1,2,1)} + 3 M_{(1,1,1,1,1)}.
\end{eqnarray*}

\begin{theorem}
  For any diagram $D$, $\Kohnert_D(X)$ is a well-defined quasisymmetric function that expands non-negatively into the monomial quasisymmetric functions.
  \label{thm:kohnert-stable}
\end{theorem}

\begin{proof}
  For $T \in \MKD(D)$, we have $0\times T \in \MKD(0 \times D)$. Moreover, from Definition~\ref{def:MKD}, it is clear that the monomial slide expansion of $\kohnert_{0^m \times D}$ is stable once $m$ is at least the number of cells of $D$. Therefore the result follows from Theorem~\ref{thm:kohnert-mono} and Theorem~\ref{thm:mono-stable}.
\end{proof}

We review two motivating examples of this stability. For $D$ a Rothe diagram for $w$, the Kohnert polynomial is a Schubert polynomial and the stable limit, as shown by Macdonald \cite{Mac91}, is the Stanley symmetric function \cite{Sta84} introduced by Stanley to enumerate reduced expressions for a permutation. Implicit in the work of Lascoux and Sch{\"u}tzenberger \cite{LS90} and explicit in \cite{AS-2}, the Kohnert polynomial of a key diagram stabilizes to the Schur function indexed by the partition to which the weak composition sorts. Both of these examples have stable limits that are symmetric functions, but Kohnert quasisymmetric functions are not always symmetric. In Section~\ref{sec:right}, we consider a new Kohnert basis that gives rise to an interesting new basis for quasisymmetric functions.

As remarked earlier, two diagrams that differ by insertion or deletion of empty columns give rise to the same Kohnert polynomial. In the stable limit, we can strengthen this with the following.

\begin{proposition}
  Given two diagrams $D$ and $D^{\prime}$ that differ by insertion or deletion of empty rows, we have $\Kohnert_D = \Kohnert_{D^{\prime}}$. 
  \label{prop:delete-row}
\end{proposition}

\begin{proof}
Fix a positive integer $m$. Then $\Kohnert_D(x_1,\ldots , x_m, 0, 0, \ldots )$ is the weighted sum of all Kohnert diagrams of $D$ whose highest cell is weakly below row $m$, and similarly for  $\Kohnert_{D^\prime}(x_1,\ldots , x_m, 0, 0, \ldots )$. Slightly abusing notation, let $\flatten(D)$ be the diagram obtained by deleting all empty rows from $D$. Then clearly any Kohnert diagram of $0^m\times D$ whose highest cell is weakly below row $m$ is also a Kohnert diagram of  $0^m\times \flatten(D)$. By definition $\flatten(D) = \flatten(D^\prime)$, therefore  $\Kohnert_D(x_1,\ldots , x_m, 0, 0, \ldots ) = \Kohnert_{D^\prime}(x_1,\ldots , x_m, 0, 0, \ldots )$. The statement then follows by letting $m\to \infty$.
\end{proof}

%%%%%%%%%%%%%%%%%%%%%%%%%%%%%%%%%%%%%%%%%%%%%%%%%%%%%%%%%%%%%%%%
%
\section{Fundamental slide expansions}
%
%%%%%%%%%%%%%%%%%%%%%%%%%%%%%%%%%%%%%%%%%%%%%%%%%%%%%%%%%%%%%%%%
\label{sec:fund}

Strengthening the results of the previous section, we next investigate the fundamental slide expansion of a Kohnert polynomial, which is not always nonnegative. In Section~\ref{sec:fund-def}, we review the basis of fundamental slide polynomials introduced in \cite{AS17}. In Section~\ref{sec:fund-FKD}, we characterize those diagrams for which the corresponding Kohnert polynomial expands nonnegatively into the fundamental slide basis. Further, we conjecture a simple condition on diagrams that ensures the corresponding Kohnert polynomial expands nonnegatively into the Demazure character basis. In Section~\ref{sec:fund-stable}, we prove the surprising fact that, while some Kohnert polynomials are not positive on the fundamental slide basis, every Kohnert quasisymmetric function is positive on the fundamental quasisymmetric function basis.

%%%%%%%%%%%%%%%%%%%%%%%%%%%%%%%%%%%%%%%%%%%%%%%%%%%%%%%%%%%%%%%%
\subsection{Fundamental slide polynomials}
%%%%%%%%%%%%%%%%%%%%%%%%%%%%%%%%%%%%%%%%%%%%%%%%%%%%%%%%%%%%%%%%
\label{sec:fund-def}

Gessel \cite{Ges84} introduced another basis for quasisymmetric functions that is closely related to Schur functions.

Given two compositions $\alpha$ and $\beta$ of the same size, say that $\beta$ \emph{refines} $\alpha$ if there exist indices $i_1 < \cdots < i_{\ell}$ such that $\beta_{i_j+1} + \cdots + \beta_{i_{j+1}} = \alpha_{j+1}$. For example, $(1,2,2)$ refines $(3,2)$ but does not refine $(2,3)$.

\begin{definition}[\cite{Ges84}]
  The \emph{fundamental quasisymmetric polynomial indexed by $\alpha$} is 
  \begin{equation}
    F_{\alpha}(x_1,\ldots,x_n) \,\, = \sum_{\beta \ \mathrm{refines} \ \alpha} \! M_{\beta} \,\,\,\,\,\,
    = \sum_{\flatten(b) \ \mathrm{refines} \ \alpha} \! x_1^{b_1} \cdots x_n^{b_n} ,
    \label{e:fundamental}
  \end{equation}
  where the latter sum is over all weak compositions of length $n$ whose flattening refines $\alpha$.
  \label{def:fundamental}
\end{definition}

For example, taking $\alpha = (2,1,2)$ and restricting to $4$ variables, we have
\begin{eqnarray*}
  F_{(2,1,2)}(x_1,x_2,x_3,x_4) & = & M_{(2,1,2)}(x_1,x_2,x_3,x_4) + M_{(1,1,1,2)}(x_1,x_2,x_3,x_4) + M_{(2,1,1,1)}(x_1,x_2,x_3,x_4), \\
  & = & x_1^2 x_2 x_3^2 + x_1^2 x_2 x_4^2 + x_1^2 x_3 x_4^2 + x_2^2 x_3 x_4^2 + x_1 x_2 x_3 x_4^2 + x_1^2 x_2 x_3 x_4.
\end{eqnarray*}

The fundamental quasisymmetric functions inspired the \emph{fundamental slide polynomials} of Assaf and Searles \cite{AS17}, analogous to the relationship between the monomial slide polynomials and the monomial quasisymmetric polynomials.

\begin{definition}[\cite{AS17}]
  The \emph{fundamental slide polynomial indexed by $a$} is
  \begin{equation}
    \fund_{a} = \sum_{\substack{b \geq a \\ \flatten(b) \ \mathrm{refines} \ \flatten(a)}} x_1^{b_1} \cdots x_n^{b_n},
    \label{e:fundamental-shift}
  \end{equation}
  where $b \geq a$ means $b_1 + \cdots + b_k \geq a_1 + \cdots + a_k$ for all $k=1,\ldots,n$.
  \label{def:fundamental-shift}
\end{definition}

For example, taking $a = (2,0,1,2)$, which implies $4$ variables, we compute
\begin{displaymath}
  \fund_{(2,0,1,2)} = \mono_{(2,0,1,2)} +\mono_{(2,1,1,1)} = x_1^2 x_2 x_3^2 + x_1^2 x_2 x_4^2 + x_1^2 x_3 x_4^2 + x_1^2 x_2 x_3 x_4.
\end{displaymath}

The fundamental slide polynomials generalize the fundamental quasisymmetric polynomials.

\begin{proposition}[\cite{AS17}]
  For a weak composition $a$ of length $n$, $\fund_a$ is quasisymmetric in $x_1,\ldots,x_n$ if and only if $a_j\neq 0$ whenever $a_i\neq 0$ for some $i<j$. Moreover, in this case we have $\fund_a = F_{\flatten(a)}(x_1,\ldots,x_n)$.
  \label{prop:lift-fund}
\end{proposition}

The fundamental slide polynomials give a basis for the polynomial ring \cite{AS17}. Remarkably, their structure constants are non-negative and generalize the shuffle product of Eilenberg and Mac Lane \cite{EM53}.

\begin{theorem}[\cite{AS17}]
  The fundamental slide polynomials $\{\fund_a\}$ are a basis of the polynomial ring with structure constants 
  \begin{equation}
    \fund_{a} \fund_{b} = \sum_{c} [c \mid a \shuffle b] \fund_{c},
  \end{equation}
  where $[c \mid a \shuffle b]$ is the coefficient of $c$ in the \emph{slide product} $a \shuffle b$. In particular, $[c \mid a \shuffle b]$ is a non-negative integer.
\end{theorem}

%Unlike Demazure characters and Schubert polynomials, they are not a Kohnert basis.

As was the case for the monomial slide polynomials, the fundamental slide polynomials are also not a Kohnert basis.

\begin{proposition}
  The fundamental slide polynomials are not a Kohnert basis.
\end{proposition}

\begin{proof}
  For any diagram $D$ of weight $(0,2,1)$, we claim $\kohnert_D \neq \fund_{(0,2,1)}$. If this is a Kohnert polynomial $\kohnert_{D}$, then $D$ must have two cells in row 2 and one in row 3. Let $D$ be any such diagram. Since there is only one cell, say $c$ in row 3 of $D$, this cell is the rightmost in its row. If $c$ is in the same column as some cell in row 2 then applying a Kohnert move to $c$ yields a diagram of weight $(1,2,0)$, otherwise applying a Kohnert move to $c$ yields a diagram of weight $(0,3,0)$. Thus $\kohnert_{D}$ must have one of $x_1 x_2^2$ or $x_2^3$ appear as a term. However, neither of these terms appears in $\fund_{(0,2,1)} = x_2^2 x_3 + x_1^2 x_3 + x_1 x_2 x_3 + x_1^2 x_2$, and so $\fund_{(0,2,1)}$ cannot be a Kohnert polynomial.  
\end{proof}

%%%%%%%%%%%%%%%%%%%%%%%%%%%%%%%%%%%%%%%%%%%%%%%%%%%%%%%%%%%%%%%%
\subsection{Fundamental diagrams}
%%%%%%%%%%%%%%%%%%%%%%%%%%%%%%%%%%%%%%%%%%%%%%%%%%%%%%%%%%%%%%%%
\label{sec:fund-FKD}

One motivation for defining and studying the fundamental slide polynomials is a refined expansion of Schubert polynomials \cite{AS17}. This formula utilizes the \emph{pipe dream} model for the monomial expansion of Schubert polynomials given by Bergeron and Billey \cite{BB93} based on the \emph{compatible sequences} model due to Billey, Jockusch, and Stanley \cite{BJS93}.

\begin{theorem}[\cite{AS17}]
  For $w$ a permutation, we have
  \begin{equation}
    \schubert_w = \sum_{P \in \mathrm{QPD}(w)} \fund_{\wt(P)},
  \end{equation}
  where the sum is over \emph{quasi-Yamanouchi pipe dreams} for $w$.
  \label{thm:schubert-slide}
\end{theorem}

For example, the Schubert polynomial for the permutation $143625$ is
\begin{eqnarray*}
  \schubert_{143625} & = & \fund_{(0, 2, 1, 2)} + \fund_{(1, 2, 0, 2)} + \fund_{(0, 2, 2, 1)} + \fund_{(0, 3, 1, 1)} + \fund_{(1, 2, 1, 1)} + \fund_{(1, 3, 0, 1)} + \fund_{(2, 2, 0, 1)}.
\end{eqnarray*}

Assaf and Searles \cite{AS-2} also show that the Demazure characters have a natural decomposition into fundamental slide polynomials. This formula utilized Kohnert's model for Demazure characters \cite{Koh91}.

\begin{theorem}[\cite{AS-2}]
  For a weak composition $a$, we have
  \begin{equation}
    \key_a = \sum_{T \in \mathrm{QKT}(a)} \fund_{\wt(T)},
  \end{equation}
  where the sum is over all \emph{quasi-Yamanouchi Kohnert tableaux} for $a$.
  \label{thm:key-slide}
\end{theorem}

For example, the Demazure character for the weak composition $(0,2,1,2)$ decomposes as
\begin{displaymath}
  \key_{(0,2,1,2)} = \fund_{(0, 2, 1, 2)} + \fund_{(1, 2, 0, 2)} + \fund_{(0, 2, 2, 1)} + \fund_{(1, 2, 1, 1)}.
\end{displaymath}

Generalizing these two examples, along with the common notion of \emph{quasi-Yamanouchi} used in both expansions, we have the following.

\begin{definition}
  For a diagram $D$, define the subset of \emph{quasi-Yamanouchi Kohnert diagrams for $D$}, denoted by $\FKD(D)$, by
  \begin{equation}
    \FKD(D) = \left\{ T \in \KD(D) \mid \begin{array}{l}
      \Up_r(T) \not\in \KD(D) \ \forall \ r \mbox{ such that all cells in } \\
      \mbox{row $r+1$ lie strictly left of all cells in row $r$}
    \end{array} \right\}.
  \end{equation}
  \label{def:FKD}
\end{definition}

Note that $\FKD(D) \subseteq \MKD(D)$. For example, for $D$ the third diagram in Figure~\ref{fig:diagrams}, Figure~\ref{fig:FKD} shows the set $\FKD(D)$. Notice that
\[ \kohnert_D = \fund_{(0,2,1,2)} + \fund_{(1,2,0,2)}, \]
which corresponds precisely to the weights of the diagrams in $\FKD(D)$. 

\begin{figure}[ht]
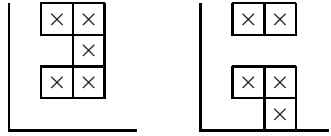

  \begin{center}
    \begin{displaymath}
      \begin{array}{c@{\hskip 2\cellsize}c}
        \vline\tableau{ & \times & \times \\  & & \times \\  & \times & \times \\ & & & \\ \hline} &
        \vline\tableau{ & \times & \times \\  & & \\  & \times & \times \\ & & \times & \\ \hline} 
      \end{array}
    \end{displaymath}
    \caption{\label{fig:FKD}The set $\FKD(D)$ of quasi-Yamanouchi Kohnert diagrams for $D$ the leftmost diagram above.}
  \end{center}
\end{figure}

Similarly, seven of the Kohnert diagrams in Figure~\ref{fig:KD-rothe} are in $\FKD(\D(143625))$ and four of the Kohnert diagrams in Figure~\ref{fig:KD-key} are in $\FKD(\D(0,2,1,2))$. The fundamental slide generating polynomials of these two sets are $\schubert_{143625}$ and $\key_{(0,2,1,2)}$, respectively.

Similar to the proof of Theorem~\ref{thm:kohnert-mono}, we wish to consolidate Kohnert diagrams into equivalence classes, each of which contains a unique quasi-Yamanouchi Kohnert diagram, so that the fundamental slide expansion of the corresponding Kohnert polynomials is precisely given by the quasi-Yamanouchi Kohnert diagrams. However, unlike the case with monomial slide polynomials, Kohnert polynomials are not, in general, fundamental slide positive. For example, taking $D$ to be the left diagram in Figure~\ref{fig:F-neg}, the corresponding Kohnert polynomial expands as
\[ \kohnert_D = \mono_{(0,1,1)} + \mono_{(0,2,0)} = \fund_{(0,1,1)} + \fund_{(0,2,0)} - \fund_{(1,1,0)}. \]

\begin{figure}[ht]
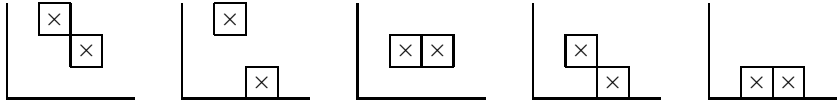

  \begin{center}
    \begin{displaymath}
      \begin{array}{c@{\hskip 1.5\cellsize}c@{\hskip 1.5\cellsize}c@{\hskip 1.5\cellsize}c@{\hskip 1.5\cellsize}c}
        \vline\tableau{ & \times & \\  & & \times \\ & & & \\ \hline} &
        \vline\tableau{ & \times & \\  & & \\ & & \times & \\ \hline} &
        \vline\tableau{ & & \\  & \times & \times \\ & & & \\ \hline} &
        \vline\tableau{ & & \\  & \times & \\ & & \times & \\ \hline} &
        \vline\tableau{ & & \\  & & \\ & \times & \times & \\ \hline} 
      \end{array}
    \end{displaymath}
    \caption{\label{fig:F-neg}The Kohnert diagrams for the leftmost diagram above, for which the Kohnert polynomial is not fundamental slide positive.}
  \end{center}
\end{figure}

The impediment to fundamental slide positivity is captured by the following notion.

\begin{definition}
  A diagram is \emph{split} if there exist rows $r_1<r_2$ and columns $c_1<c_2$ such that there are cells in positions $(r_2,c_1)$ and $(r_1,c_2)$ but no cells in rows $r$ for $r_1 < r < r_2$ and no cells in positions $(r_1,c)$ for $c<c_2$ or $(r_2,c)$ for $c>c_1$. In this case, we call the cells $(r_2,c_1)$ and $(r_1,c_2)$ a \emph{split pair}.
  \label{def:diagram-split}
\end{definition}

That is, a diagram is split if it contains two cells with one strictly northwest of the other such that no other cells lie between them in the reading order that reads left to right along rows, starting with the highest row. For example, the first, second and fourth diagrams in Figure~\ref{fig:F-neg} are split, and neither of the diagrams in Figure~\ref{fig:FKD} is split. Indeed, none of the diagrams in $\FKD(\D(143625))$ nor in $\FKD(\D(0,2,1,2))$ is split.

\begin{lemma}
  Let $D$ be a diagram such that no diagram in $\FKD(D)$ is split. Let $U \in \KD(D)$ such that both $S = \Up_{i_k} \cdots \Up_{i_1} (U)$ and $T = \Up_{j_l} \cdots \Up_{j_1} (U)$ are lifts that raise rows only when all cells of the row above lie strictly to the left. If $S,T\in\FKD(D)$, then $S=T$.
  \label{lem:destand}
\end{lemma}

\begin{proof}
  Since each lift either moves a row up or consolidates two rows, for $U \in \KD(D)$, if $U$ has a path to $T \in \FKD(D)$, then the rows of $T$ are unions of the rows of $U$, taken in order and moved weakly up. Suppose that $U$ has another lift path to $S$. Consider the highest row, say $r_2$, in which $S$ and $T$ differ, say with $T$ having $t$ cells and $S$ having $s$ cells in row $r_2$. Without loss of generality, we may assume $s<t$. Then $T$ must have consolidated more rows of $U$ into its row $r_2$. Therefore row $r_2$ of $S$ must consist of the $s$ leftmost cells of row $r_2$ of $T$. Set $c_1$ to be the column of the $s$th cell from the left in row $r_2$ of $T$ (equivalently, $S$), and let $c_2$ be the column of the next cell to its right. Let $r_1 < r_2$ be the highest nonempty of $S$ below $r_2$. Then the leftmost cell of $r_1$ in $S$ must lie in column $c_2$. In particular, $S$ has cells in positions $(r_2,c_1)$ and $(r_1,c_2)$ with no cells in between, and so $S$ is split. Therefore $S$ must equal $T$.  
\end{proof}

The conclusion of Lemma~\ref{lem:destand} does not always hold. For example, taking $D$ to be the leftmost diagram in Figure~\ref{fig:F-neg}, consider $U$ to be the fourth diagram from the left. Then $T$ has two lifting paths, namely $\Up_1 \Up_2 (T)$ which terminates in the leftmost diagram and $\Up_1(T)$ which terminates in the third diagram, both of which are quasi-Yamanouchi.

While Lemma~\ref{lem:destand} is sufficient to guarantee that lifting paths converge to the same quasi-Yamanouchi diagram, it is not tight. For instance, if we take $D$ to be the fourth diagram from the left in Figure~\ref{fig:F-neg}, then $\KD(D)$ consists of the fourth and fifth diagrams, neither of which can lift. Though $D$ is split, the conclusion of Lemma~\ref{lem:destand} trivially holds. However, this is somewhat accidental since shifting $D$ up one row gives the leftmost diagram in Figure~\ref{fig:F-neg} which we have just seen fails to have well-defined quasi-Yamanouchi lifts. 

\begin{definition}
  A diagram $D$ is \emph{fundamental} if for each cell $(r,c)$ of $D$ that is leftmost in its row, either there is a cell in position $(r+1,c)$, or for each column $c'<c$ and for all $k\geq 1$ we have
  \begin{equation}
    \#\{(s,c') \in D \mid r < s \leq r+k \} \leq \#\{(s,c) \in D \mid r < s \leq r+k \}.
    \label{e:diagram-fund}
  \end{equation}
  \label{def:diagram-fund}
\end{definition}

For example, see Figure~\ref{fig:fund_diagrams}. The diagram on the left is fundamental. The diagram on the right is not fundamental: the box in row $2$, column $4$ fails the condition with respect to the second column when $k=3$.

\begin{figure}[ht]
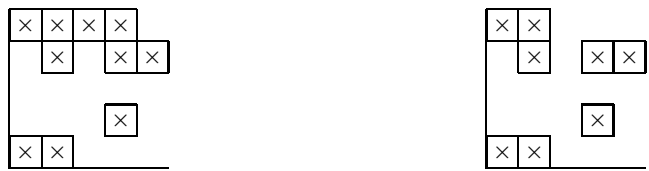

  \begin{center}
    \begin{displaymath}
      \begin{array}{c@{\hskip 10\cellsize}c}
        \vline\tableau{\times & \times & \times & \times \\ & \times & & \times & \times \\ & \\ & &  & \times \\  \times & \times  \\ \hline} &
        \vline\tableau{\times & \times & \\ & \times & & \times & \times \\ & \\ & &  & \times \\  \times & \times  \\ \hline} 
      \end{array}
    \end{displaymath}
    \caption{\label{fig:fund_diagrams}An example (left) and non-example (right) of the fundamental property.}
  \end{center}
\end{figure}

In both key diagrams and Rothe diagrams, a cell that is leftmost in its row cannot have any cell strictly above and strictly left, whence both are examples of fundamental diagrams. The significance of Definition~\ref{def:diagram-fund} is that fundamental diagrams are precisely those satisfying the hypothesis of Lemma~\ref{lem:destand}. To prove this, we begin with the following.

\begin{lemma}
  If $D$ has a split quasi-Yamanouchi Kohnert diagram, then $D$ has a split quasi-Yamanouchi Kohnert diagram in which the bottom-right cell of the split pair is in its original row in $D$. %(i.e., it did not move when passing from $D$ to this FKD).
\label{lem:bottomright}
\end{lemma}

\begin{proof}
  Take any split $T \in \FKD(D)$. By definition, \emph{some} cell in the row of the bottom-right cell of the split pair is in its original row in $D$, otherwise this row could lift in $T$, contradicting that $T$ is a quasi-Yamanouchi Kohnert diagram. Let $c$ be the leftmost such cell. If $c$ is not leftmost in its row, then from left to right, perform reverse Kohnert moves (no jumps) on the cells to the left of $c$ in the row of $c$ until they either reach their original row or land to the right of an existing cell, whichever happens first. This process creates another quasi-Yamanouchi Kohnert diagram for $D$, and, in particular, it is split on cell $c$ and the cell that was closest (on the left) to $c$ in $T$.
\end{proof}

\begin{theorem}
  If a diagram $D$ is fundamental, then no quasi-Yamanouchi Kohnert diagram for $D$ is split. Conversely, if no quasi-Yamanouchi Kohnert diagram for $0^{|D|} \times D$ is split, then $D$ is fundamental.
  \label{thm:split}
\end{theorem}

\begin{proof}
  First suppose that $D$ is not fundamental. We will construct a split quasi-Yamanouchi Kohnert diagram of $0^{|D|}\times D$. We may assume there are no cells strictly above and strictly right of the cell $(r,c)$ in $0^{|D|}\times D$: if there are, then one can push all such cells down to be weakly below row $r$, so that the top-left of all these cells now lies in row $r$. These cells are now all ``anchored'' on the cell in position $(r,c)$, so the result is a quasi-Yamanouchi Kohnert diagram of $0^{|D|}\times D$. 
  
 Select the rightmost column $c'<c$ and then smallest $k$ such that the condition fails. In what follows, we work entirely in the rectangle between rows $r$ and $r+k$ (inclusive) and between columns $c'$ and $c$ (inclusive). Working from right to left, push cells in all columns $c', c'+1, \ldots c-1$ downwards (jumping over other cells if necessary) so that each of these cells ends up in the same row as some cell in column $c$. This is possible since the condition is satisfied on these columns. Now, all rows $r+1, \ldots , r+k$ that have no cell in column $c'$ also have no cell in any column to the right of $c'$. Since we never moved any cell in column $c$, these rows cannot be de-standardized, so the result is a quasi-Yamanouchi Kohnert diagram of $0^{|D|}\times D$.

Now consider the top cell $C$ of column $c'$ (in rows weakly lower than $r+k$). This cell $C$ has no cell to its right in its row, in particular, there is no cell in the same row in column $c$, since this would contradict the minimality of $k$. Also by minimality of $k$ and the argument of the previous paragraph, every row strictly between row $r$ and the row of $C$ either has cells in both columns $c'$ and $c$, or neither. To complete the construction, take the cell $C$ and move it downwards, jumping over all cells in column $c'$, rows $r+1, \ldots , r+k$. The cell $C$ ends in the row immediately below the lowest cell of column $c$ that is above $(r,c)$, which, by assumption, is in row $r+2$ or higher. By construction this is a quasi-Yamanouchi Kohnert diagram for $0^{|D|}\times D$, and it is split over the cell $C$ and the cell in position $(r,c)$.

For the other direction, suppose the $D$ satisfies the fundamental condition. We claim the cell $(r,c)$ can never be the lower-right cell in the split pair of a split quasi-Yamanouchi Kohnert diagram of $D$. In particular $D$ has no split quasi-Yamanouchi Kohnert diagram where the bottom-right cell is in its original position, so by Lemma~\ref{lem:bottomright} $D$ has no split quasi-Yamanouchi Kohnert diagrams at all. To prove this claim, we need to show that no cell strictly above and left of $(r,c)$ in $D$ can ever get strictly below the lowest cell of $D$ in the positions $(r+1,c), (r+2,c), \ldots (r+k,c)$ via a series of Kohnert moves, while remaining strictly above row $r$. Clearly this is true if $D$ has a cell in position $(r+1,c)$. Otherwise, consider any column $c'$ strictly left of column $c$. The condition states that for any cell in column $c$ (above row $r$), there are at least as many cells of $D$ weakly below this cell in column $c$ as there are in column $c'$. Performing Kohnert moves on the cells of column $c$ does not alter this, so we may suppose we do not perform any Kohnert moves on column $c$. Moreover, for simplicity, we may assume there are no cells of $D$ in columns between $c'$ and $c$ and above row $r$, since any such cells are to the right of column $c'$ and so only impede cells of column $c'$ from moving downwards. Suppose we can perform a Kohnert move on a cell $C$ of column $c'$. This means there is no cell in column $c$ in the row of $C$, therefore the condition implies there are strictly more cells in column $c$ strictly below $C$ than there are in column $c'$ strictly below $C$. 

Now perform the Kohnert move on $C$. The condition is clearly preserved for cells of $c'$ above the original position of $C$ and strictly below the position where $C$ lands. For cell $C$ itself, the condition is preserved since if $C$ jumps over, say, $\ell$ cells for some $\ell \ge 0$, then $C$ necessarily moves from being strictly above to strictly below weakly fewer than $\ell$ cells of column $c$ (with equality only if the all of the first $\ell$ positions in column $c$ strictly below the row of $C$ are occupied by cells). Finally, consider any cell that $C$ jumps over. By definition, all such cells exist in a column interval immediately below $C$. Therefore, if any one of these cells, say $X$, met the condition with equality, $C$ would necessarily fail the condition since there are strictly more cells in column $c'$ than there are in column $c$, in the rows between the row of $X$ and the row of $C$ (inclusive). Therefore, for any such cell $X$, there are strictly more cells of $D$ weakly below $X$ in column $c$ than there are in column $c'$ (above row $r$). Thus when $C$ moves from above to below $X$, the condition is maintained on $X$. Therefore, the condition is preserved under Kohnert moves on column $c'$, which in particular implies that no cell strictly above and left of $(r,c)$ in $D$ can ever get strictly below the lowest cell of $D$ in the positions $(r+1,c), (r+2,c), \ldots (r+k,c)$.
\end{proof}

Combining Lemma~\ref{lem:destand} and Theorem~\ref{thm:split}, we may give the fundamental slide expansion of a Kohnert polynomial indexed by a fundamental diagram.

\begin{theorem}
  Given a fundamental diagram $D$, we have
  \begin{equation}
    \kohnert_D = \sum_{T \in \FKD(D)} \fund_{\wt(T)} .
  \label{e:kohnert-slide}
  \end{equation}
  In particular, these Kohnert polynomials expand non-negatively into fundamental slide polynomials.
  \label{thm:kohnert-slide}
\end{theorem}

\begin{proof}
  By Lemma~\ref{lem:destand}, for any $U\in\KD(D)$, we may define the de-standardization of $U$, denoted by $\destand_D(U)$, to be the result of any maximal length lifting path that necessarily results in a quasi-Yamanouchi Kohnert diagram for $D$. If $\destand_D(U)=T$, then $\wt(U) \geq \wt(T)$ and $\flatten(\wt(U))$ refines $\flatten(\wt(T))$ since $T$ is obtained recursively by moving \emph{all} cells in row $i-1$ of $U$ to row $i$ in $U$. Conversely, we claim that given $T \in \FKD(D)$, for every weak composition $b$ of length $n$ such that $b \geq \wt(T)$ and $\flatten(b)$ refines $\flatten(\wt(T))$, there is a unique $U \in \KD(D)$ with $\wt(U) = b$ such that $\destand_D(U) = T$. From the claim, we have
  \begin{displaymath}
    \sum_{U \in \destand_D^{-1}(T)} x^{\wt(U)} = \fund_{\wt(T)},
  \end{displaymath}
  from which theorem follows. To construct $U$ from $b$ and $T$, for $j = 1,\ldots,n$, if $\wt(T)_{j} = b_{i_{j-1} + 1} + \cdots + b_{i_{j}}$, then, from right to left, move the first $b_{i_{j-1} + 1}$ cells down to row $i_{j-1} + 1$, the next $b_{i_{j-1} + 2}$ cells down to row $i_{j-1} + 2$, and so on. Each of these moves is a valid Kohnert move with no cells jumping over any others. Existence is proved, and uniqueness follows from the lack of choice at every step.
\end{proof}

Both Schubert polynomials and Demazure characters expand non-negatively into fundamental slide polynomials, with the former indexed by quasi-Yamanouchi pipe dreams and the latter by quasi-Yamanouchi Kohnert tableaux. Since both Rothe diagrams and key diagrams are fundamental, the expansion in \eqref{e:kohnert-slide} gives a common generalization of these results. 

%%%%%%%%%%%%%%%%%%%%%%%%%%%%%%%%%%%%%%%%%%%%%%%%%%%%%%%%%%%%%%%%
\subsection{Kohnert quasisymmetric functions are fundamental positive}
%%%%%%%%%%%%%%%%%%%%%%%%%%%%%%%%%%%%%%%%%%%%%%%%%%%%%%%%%%%%%%%%
\label{sec:fund-stable}

Assaf and Searles \cite{AS17} showed that the fundamental slide polynomials stabilize and that their stable limits are precisely the fundamental quasisymmetric functions.

\begin{theorem}[\cite{AS17}]
  For a weak composition $a$, we have
  \begin{equation}
    \lim_{m\rightarrow\infty}\fund_{0^m \times a} = F_{\flatten(a)}(x_1,x_2,\ldots),
  \end{equation}
  where $0^m \times a$ denotes the weak composition obtained by prepending $m$ $0$'s to $a$.
  \label{thm:fund-stable}
\end{theorem}

In order to remove extraneous redundancy from the stable limits, we say that a diagram is \emph{flat} if there is no empty row below a nonempty row. For each diagram $D$, we define $\flatten(D)$ to be the diagram obtained by removing empty rows. For any diagrams $C,D$ such that $\flatten(C)=\flatten(D)$, it is clear that $\Kohnert_D(X) = \Kohnert_C(X)$. In particular, in the stable limit, it is enough to consider flat diagrams. It is an interesting, though clearly difficult, question to characterize when two Kohnert quasisymmetric functions are equal. We offer the following partial solution that allows us to consider only flattened, fundamental diagrams.

\begin{lemma}
  Given a diagram $D$ and a row index $i$ such that all cells in row $i+1$ lie strictly left of all cells in row $i$, we have $\Kohnert_D(X) = \Kohnert_{\Up_i(D)}(X)$. 
\label{lem:flat}
\end{lemma}

\begin{proof}
It is enough to show $\kohnert_{0^{m + |D| + 1}\times D}(x_1, \ldots x_{m}) = \kohnert_{0^{m + |D| + 1}\times \Up_i(D)}(x_1, \ldots x_{m})$ for all $m$. Specifically, let $\KD_i(D)$ denote the subset of $\KD(D)$ consisting of diagrams having no cells in row $i+1$ or higher. We will show $\KD_m(0^{m + |D| + 1}\times D) = \KD_m(0^{m + |D| + 1}\times \Up_i(D))$.

Since  $0^{m + |D| + 1}\times D$ is a Kohnert diagram of $0^{m + |D| + 1}\times \Up_i(D)$, it is clear that $\KD_m(0^{m + |D| + 1}\times D) \subset \KD_m(0^{m + |D| + 1}\times \Up_i(D))$. To see the other containment, observe that $0^{m + |D|}\times \Up_i(D)$ is a Kohnert diagram of $0^{m + |D| + 1}\times D$, formed by dropping all cells in row $i+1$ of $D$ down to row $i$ and dropping all cells not in rows $i$ or $i+1$ down one row.
 \end{proof}

\begin{figure}[ht]
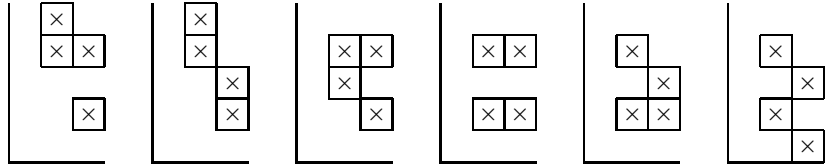

  \begin{center}
    \begin{displaymath}
      \begin{array}{c@{\hskip 1.5\cellsize}c@{\hskip 1.5\cellsize}c@{\hskip 1.5\cellsize}c@{\hskip 1.5\cellsize}c@{\hskip 1.5\cellsize}c}
        \vline\tableau{ & \times & \\ & \times & \times \\ & \\ & & \times \\ & \\ \hline} &
        \vline\tableau{ & \times & \\ & \times & \\ & & \times \\ & & \times \\ & \\ \hline} &
        \vline\tableau{ & & \\ & \times & \times \\ & \times \\ & & \times \\ & \\ \hline} &
        \vline\tableau{ & & \\ & \times & \times \\ & \\ & \times & \times \\ & \\ \hline} &
        \vline\tableau{ & & \\ & \times & \\ & & \times \\ & \times & \times \\ & & \\ \hline} &
        \vline\tableau{ & & \\ & \times & \\ & & \times \\ & \times & \\ & & \times \\ \hline} 
      \end{array}
    \end{displaymath}
    \caption{\label{fig:F-not-fund}The elements of $\MKD(D)$ for $D$ the leftmost diagram, which is not fundamental.}
  \end{center}
\end{figure}

For example, letting $D$ be the leftmost diagram in Figure~\ref{fig:F-not-fund}, we may compute the fundamental quasisymmetric function expansion of the Kohnert quasisymmetric function by
\begin{eqnarray*}
  \Kohnert_D & = & F_{\flatten(0,1,0,2,1)} + F_{\flatten(0,2,0,2,0)} + F_{\flatten(0,1,1,2,0)} - F_{\flatten(1,1,0,2,0)} \\
  & = & F_{(1,2,1)} + F_{(2,2)}.
\end{eqnarray*}
Notice that each of the split quasi-Yamanouchi Kohnert diagram is canceled in the limit. Correspondingly, if we lift $D$ to a fundamental diagram, we obtain the fourth diagram in Figure~\ref{fig:F-stable} for which there are two quasi-Yamanouchi Kohnert diagrams. 

\begin{figure}[ht]
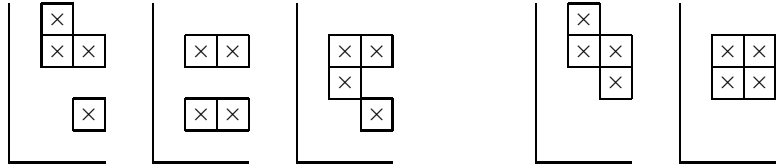

  \begin{center}
    \begin{displaymath}
      \begin{array}{c@{\hskip 1.5\cellsize}c@{\hskip 1.5\cellsize}c@{\hskip 1.5\cellsize}c@{\hskip 3\cellsize}c@{\hskip 1.5\cellsize}c}
        \vline\tableau{ & \times & \\ & \times & \times \\ & \\ & & \times \\ & \\ \hline} &
        \vline\tableau{ & & \\ & \times & \times \\ & \\ & \times & \times \\ & \\ \hline} & 
        \vline\tableau{ & & \\ & \times & \times \\ & \times \\ & & \times \\ & \\ \hline} & &
        \vline\tableau{ & \times & \\ & \times & \times \\ & & \times \\ & \\ & \\ \hline} &
        \vline\tableau{ & & \\ & \times & \times \\ & \times & \times \\ & \\ & \\ \hline} 
      \end{array}
    \end{displaymath}
    \caption{\label{fig:F-stable}The three quasi-Yamanouchi Kohnert diagrams for the leftmost diagram (left), which is not itself fundamental, and the two quasi-Yamanouchi Kohnert diagrams for the fourth diagram (right), which is fundamental.}
  \end{center}
\end{figure}

Despite the restriction of Theorem~\ref{thm:kohnert-slide} to fundamental diagrams, Lemma~\ref{lem:flat} allows us to prove that positivity in the stable limit holds in general.

\begin{theorem}
  For any diagram $D$ and any $m$ at least as great as the number of cells of $D$, we have
  \begin{equation}
    \Kohnert_D = \sum_{\substack{T \in \FKD(0^m \times D) \\ T \ \mathrm{not  \ split}}} F_{\flatten(\wt(T))} .
  \label{e:kohnert-gessel}
  \end{equation}
  In particular, Kohnert quasisymmetric functions expand non-negatively into fundamental quasisymmetric functions.
  \label{thm:kohnert-stable-fund}
\end{theorem}

\begin{proof}
  First consider the case when $D$ is fundamental. In this case, $0^m \times D$ is also fundamental for any $m \geq 0$. Therefore, if $D$ is fundamental, then by Theorem~\ref{thm:kohnert-slide}, $\kohnert_{0^m \times D}$ has a nonnegative fundamental slide expansion indexed by $\FKD(0^m \times D)$. From Definition~\ref{def:FKD}, it is clear that, for $m$ at least the number of cells of $D$, no $T \in \FKD(0^{m+1} \times D)$ has a cell in the bottom row. Therefore the fundamental slide expansion of $\kohnert_{0^m \times D}$ is stable once $m$ is at least the number of cells of $D$. Therefore by Theorem~\ref{thm:fund-stable}, the Kohnert quasisymmetric function $\Kohnert_D$ is given by
  \[ \Kohnert_D(X) = \sum_{T \in \FKD(0^m \times D)} F_{\flatten(\wt(T))} (X), \]
  for any $m$ at least as great as the number of cells of $D$ (and often much smaller). 

 If $D$ is not fundamental, then one can apply $\Up_i$ until arriving at a diagram that is fundamental. By Lemma~\ref{lem:flat}, they have the same Kohnert quasisymmetric function, which, by the argument of the previous paragraph, expands non-negatively into fundamental quasisymmetric functions.
  
To show that the expansion is indexed exactly by the non-split quasi-Yamanouchi Kohnert diagrams of $D$, use $\Up_i$ to de-standardize elements of $\KD_k(0^m \times D)$ for any $k$ and $m$ where $|D|\le k \le m$. Then since every cell of $0^m \times D$ is above the $k$th row, repeated application of $\Up_i$ (staying within $\KD_k(0^m\times D)$) always yields (a downward translation of) a non-split quasi-Yamanouchi Kohnert diagram of $0^m \times D$, and (a downward translation of) every non-split quasi-Yamanouchi Kohnert diagram of $0^m \times D$ is obtained in this way. Hence $\kohnert_{0^m \times D}(x_1, \ldots , x_k)$ is the sum of fundamental slide polynomials indexed by translations of the non-split quasi-Yamanouchi Kohnert diagram of $0^m \times D$. In the limit the weights of the translations flatten to the weights of the quasi-Yamanouchi Kohnert diagrams of $0^m \times D$, and the statement follows.
\end{proof}

Since the fundamental quasisymmetric expansion is governed by the non-split quasi-Yamanouchi diagrams, we have the following converse to Theorem~\ref{thm:kohnert-slide} using Theorems~\ref{thm:split} and \ref{thm:kohnert-stable-fund}.
  
\begin{corollary}
  Given a diagram $D$ that is not fundamental, the Kohnert polynomial $\kohnert_{0^{|D|} \times D}$ is not non-negative on the fundamental slide basis.
\end{corollary}

Theorem~\ref{thm:kohnert-stable-fund} is the strongest result one can expect in that we know of no other bases for quasisymmetric functions on which all Kohnert quasisymmetric functions are nonnegative.

%%%%%%%%%%%%%%%%%%%%%%%%%%%%%%%%%%%%%%%%%%%%%%%%%%%%%%%%%%%%%%%%
%
\section{Demazure character expansions}
%
%%%%%%%%%%%%%%%%%%%%%%%%%%%%%%%%%%%%%%%%%%%%%%%%%%%%%%%%%%%%%%%%
\label{sec:demazure}

Our two motivating examples for general Kohnert polynomials are Schubert polynomials and Demazure characters. Generalizing Proposition~\ref{prop:vex} that characterizes when a Schubert polynomial is equal to a Demazure character, Lascoux and Sch{\"u}tzenberger \cite{LS90} proved that Schubert polynomials always expand as a nonnegative integral sum of Demazure characters. Thus it is natural to explore the question of when a general Kohnert polynomial expands as a nonnegative integral sum of Demazure characters.

%%%%%%%%%%%%%%%%%%%%%%%%%%%%%%%%%%%%%%%%%%%%%%%%%%%%%%%%%%%%%%%%
\subsection{Demazure diagrams}
%%%%%%%%%%%%%%%%%%%%%%%%%%%%%%%%%%%%%%%%%%%%%%%%%%%%%%%%%%%%%%%%
\label{sec:dem-diagram}

Inspired by these two important bases, Schubert polynomials and Demazure characters, we have the following simple condition on diagrams.

\begin{definition}
  A diagram $D$ is \emph{demazure} if whenever a pair of cells $(r_2, c_1)$ and $(r_1, c_2)$ are in $D$, where $r_1<r_2$ and $c_1<c_2$, then the cell $(r_1, c_1)$ is also in $D$.
  \label{def:diagram-dem}
\end{definition}

For example, the diagram on the left of Figure~\ref{fig:dem-ex} is not demazure since it contains cells in positions $(4,2)$ and $(2,4)$ but not the cell in position $(2,2)$. In contrast, the diagram on the right of Figure~\ref{fig:dem-ex} is demazure, precisely because this impediment has been removed.

\begin{figure}[ht]
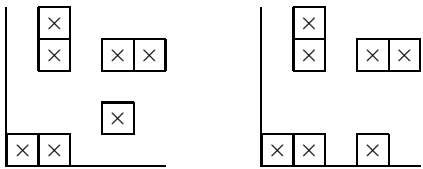

  \begin{displaymath}
    \begin{array}{c@{\hskip 3\cellsize}c}
      \vline\tableau{ & \times & \\ & \times & & \times & \times \\ & \\ & &  & \times \\ \times & \times & &  \\ \hline} &
      \vline\tableau{ & \times & \\ & \times & & \times & \times \\ & \\ & &  & \\ \times & \times & & \times  \\ \hline} 
    \end{array}
  \end{displaymath}
    \caption{\label{fig:dem-ex}The left diagram is not demazure; the right is demazure.}
\end{figure}

In particular, a demazure diagram is necessarily fundamental. Conversely, a fundamental diagram need not be demazure. For example, all four diagrams in Figure~\ref{fig:diagrams} are fundamental, but the third is not demazure. 

The following equivalent characterization of the Demazure condition is similar to Definition~\ref{def:diagram-fund} characterizing fundamental diagrams, but now we consider all cells of $D$, and strengthen the weak inequality on the column counts to a strict inequality. 

\begin{proposition}
  A diagram $D$ is demazure if and only if for each cell $(r,c)$ of $D$, for all $k\geq 1$ we have
  \begin{equation}
    \#\{(s,c') \in D \mid r < s \leq r+k \} < \#\{(s,c) \in D \mid r < s \leq r+k \},
    \label{e:diagram-dem}
  \end{equation}
  (whenever the left hand side is nonzero) for each column $c_0 < c' < c$ where $(r,c_0)$ is the rightmost cell that lies left of $(r,c)$ in row $r$, or $c_0=0$ if no such cell exists.
\end{proposition}

\begin{proof}
If $D$ is demazure, then $D$ clearly satisfies this condition since for any empty space to the left of a given cell $(r,c)$, there is no cell in the column above row $r$ in the column of that empty space.

Conversely, suppose $D$ satisfies this condition. Suppose there are two cells $(r_2, c_1)$ and $(r_1, c_2)$ of $D$ such that there are no other cells $(r,c)$ of $D$ where $c_1 \le c \le c_2$ and $r_1 \le r \le r_2$, except possibly at $(r_1, c_1)$ and $(r_2,c_2)$. In this situation, we say the pair $(r_2, c_1)$ and $(r_1, c_2)$ has no internal cells. Then the demazure condition forces the existence of a cell of $D$ at $(r_1, c_1)$. 

We induct on the sum of dimensions of any rectangle of positions in $\mathbb{N}\times \mathbb{N}$. By the argument above, the statement has to hold for any rectangle whose length and width sum to $4$ or less, establishing the base case. Given a rectangle $R$ of positions between rows $r_1 < r_2$ and columns $c_1 < c_2$, suppose $(r_2, c_1)$ and $(r_1, c_2)$ are cells of $D$ but $(r_1, c_1)$ is not. Then $(r_2, c_1)$ and $(r_1, c_2)$ must have internal cells. Let $r_a$ be the lowest row in this rectangle such that $(r_a, c_1)$ is a cell of $D$. By assumption, $a>1$. Then let $r_b$ be the highest row in this rectangle strictly below $r_a$ that has at least one cell. This exists since $r_1$ has at least one cell. Take the rightmost cell $X$ in row $r_a$ that is strictly left of the leftmost cell $Y$ in row $r_b$. By construction, the pair of cells $X, Y$ has no internal cells and there is no cell in the column of $X$ and row $r_b$ of $Y$. Moreover, since $(r_2, c_1)$ and $(r_1, c_2)$ has at least one internal cell, by construction at least one of $X$ and $Y$ is such an internal cell, and thus $X$ and $Y$ determine a rectangle whose dimension sum is strictly smaller than that of $R$, yielding the desired contradiction. 
\end{proof}

As our motivation for the demazure condition comes from Schubert polynomials and Demazure characters, we have the following motivating observation.

\begin{proposition}
  Both composition diagrams and Rothe diagrams are demazure.
\end{proposition}

\begin{proof}
  Failure of the demazure condition ensures the existence of a pair of cells $(r_1,c_2)$ and $(r_2,c_1)$ in $D$, where $r_1<r_2$ and $c_1<c_2$, such that the position $(r_1,c_1)$ is not a cell of $D$. Such a configuration is impossible in a composition diagram $D$ since all rows are left rectified, so $(r_1,c_2) \in D$ implies $(r_1,c_1) \in D$. Similarly, in a Rothe diagram $D$ , if a cell $(r_1,c_1) \not\in D$ then either $(r_2,c_1) \not\in D$ for all $r_2>r_1$ or $(r_1,c_2)\not\in D$ for all $c_2>c_1$. Thus both composition diagrams and Rothe diagrams are demazure.
\end{proof}

This together with extensive computations supports the following conjecture.

\begin{conjecture}
  Given a demazure diagram $D$, the Kohnert polynomial $\kohnert_D$ expands non-negatively into Demazure characters.
  \label{conj:demazure}
\end{conjecture}

If true, Conjecture~\ref{conj:demazure} gives an enormous class of Kohnert polynomials that have representation-theoretic and geometric significance. In further support of the conjecture, the demazure condition is exactly the same as the \emph{northwest} condition of Reiner and Shimozono \cite{RS95-2,RS98} in their study of Specht modules associated to diagrams. This suggests a connection between flagged Weyl modules and Kohnert polynomials.

In light of Proposition~\ref{prop:key-stable}, which states that Demazure characters stabilize to Schur functions, Conjecture~\ref{conj:demazure} gives an enormous class of Kohnert quasisymmetric functions that are \emph{Schur positive}. This would be striking given that not all Kohnert quasisymmetric functions are symmetric let alone Schur positive and shows the potential power of Kohnert polynomials.

%%%%%%%%%%%%%%%%%%%%%%%%%%%%%%%%%%%%%%%%%%%%%%%%%%%%%%%%%%%%%%%%
\subsection{Skew polynomials}
%%%%%%%%%%%%%%%%%%%%%%%%%%%%%%%%%%%%%%%%%%%%%%%%%%%%%%%%%%%%%%%%
\label{sec:dem-skew}

We now give a new and nontrivial example of a demazure Kohnert basis that further supports Conjecture~\ref{conj:demazure}.

\begin{definition}
  For a weak composition $a$, the \emph{skew diagram} $\mathbb{S}(a)$ is constructed as follows: 
  \begin{itemize}
  \item left justify $a_i$ cells in row $i$,
  \item for $j$ from $1$ to $n$ such that $a_j>0$, take $i<j$ maximal such that $a_i>0$, and if $a_i > a_j$, then shift rows $k \geq j$ rightward by $a_i-a_j$ columns,
    %For each $i=1,2,\ldots$ in order, shift row $i$ rightward by the minimal amount needed to ensure that whenever there are cells in positions $(r_2,c_1)$ and $(r_1,c_2)$ (for $r_1<r_2\le i$ and $c_1<c_2$), there are also cells in positions $(r_1,c_1)$ \emph{and} $(r_2,c_2)$.  %shift rows right until row ends increase bottom to top,
  \item shift each row $j$ rightward by $\#\{i<j \mid a_i=0\}$ columns.
  \end{itemize}
  The \emph{skew polynomial} is the Kohnert polynomial $\kohnert_{\mathbb{S}(a)}$.
  \label{def:diagram-skew}
\end{definition}

For example, we construct the skew diagram $\mathbb{S}(1,0,3,2,0,3)$ from the composition diagram $\D(1,0,3,2,0,3)$ by shifting rows $k \geq 4$ rightward by $1$ column since $a_3 - a_4 = 1$, then shifting rows $3,4$ rightward by one since $a_2=0$ and row $6$ rightward by two since $a_2 = a_5 = 0$. These steps are illustrated in Figure~\ref{fig:skew}.

\begin{figure}[ht]
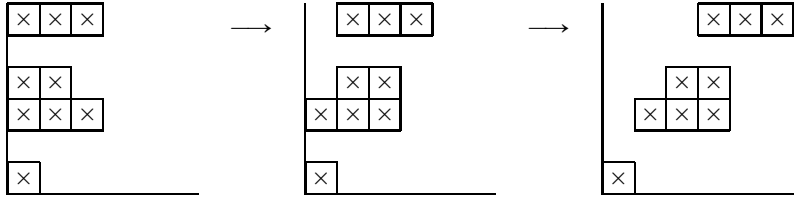

    \begin{displaymath}
      \begin{array}{c@{\hskip \cellsize}c@{\hskip 1\cellsize}c@{\hskip \cellsize}c@{\hskip 1\cellsize}c}
        \vline\tableau{ \times & \times & \times & & & \\ \\ \times & \times \\ \times & \times & \times \\ \\ \times \\\hline} & \longrightarrow &
      \vline\tableau{ & \times & \times & \times & & \\ \\ & \times & \times \\ \times & \times & \times \\ \\ \times \\\hline} & \longrightarrow &
      \vline\tableau{ & & & \times & \times & \times \\ \\ & & \times & \times \\ & \times & \times & \times \\ \\ \times \\\hline}
      \end{array}
    \end{displaymath}
    \caption{\label{fig:skew}Illustration of construction of the skew diagram $\mathbb{S}(1,0,3,2,0,3)$.}
\end{figure}

\begin{proposition}
  Skew diagrams are demazure.
\end{proposition}

\begin{proof}
  Failure of the demazure condition ensures the existence of a pair of cells $(r_1,c_2)$ and $(r_2,c_1)$ in $D$, where $r_1<r_2$ and $c_1<c_2$, such that the position $(r_1,c_1)$ is not a cell of $D$. Such a configuration is impossible in a skew diagram $D$ since rows of cells have no internal gaps and the leftmost cell in a lower row is weakly left of the leftmost cell in a higher row. Thus $(r_2,c_1) \in D$ implies $(r_1,c_1) \in D$ whenever row $r_1$ is nonempty.
\end{proof}

Conjecture~\ref{conj:demazure} implies that skew polynomials should expand nonnegatively in Demazure characters. Indeed, this fact follows using the machinery of weak dual equivalence developed in \cite{Ass-W}.

\begin{theorem}
  Skew polynomials $\{\kohnert_{\mathbb{S}(a)}\}$ form a basis of $\mathbb{Z}[x_1,x_2,\ldots,x_n]$, expand nonnegatively in demazure characters, and stabilize to Schur positive symmetric functions.
\end{theorem}

\begin{proof}
  By Theorem~\ref{thm:kohnert-basis}, since skew polynomials are a Kohnert basis, they are lower uni-triangular with respect to monomials, and as such form a basis for the polynomial ring.   
  
  In \cite{Ass-W}(Theorem~4.10), Assaf proves a generalized Littlewood--Richardson rule that expands a \emph{skew key polynomial}, indexed by a composition with a partition shape removed from the northwest corner, as a nonnegative integral sum of Demazure characters (therein called key polynomials). The definition for these skew key polynomials, \cite{Ass-W}(Definition~4.7), uses standard key tableaux. The bijection between standard key tableaux and quasi-Yamanouchi Kohnert tableaux stated in \cite{Ass-W}(Definition~3.14) and proved in \cite{Ass-W}(Theorem~3.15) together with the bijection from the latter to quasi-Yamanouchi Kohnert diagrams stated in \cite{AS-2}(Definition~2.5) and proved in \cite{AS-2}(Theorem~2.8) establishes the equivalence of skew key polynomials with the Kohnert polynomials defined by the indexing shape. Thus each skew polynomial is a skew key polynomial, and so expands nonnegatively into Demazure characters.

  Finally, by Proposition~\ref{prop:keytoSchur} Demazure characters stabilize to Schur functions, hence the stable limit of a skew polynomial is a Schur-positive symmetric function.
\end{proof}

For example, the skew diagram $\mathbb{S}(1,0,3,2,0,3)$ can be realized as the composition diagram $\D(1,0,4,4,3,6)$ skewed by the partition $(0,0,1,2,3,3)$, and so the polynomial equivalence states
\[ \kohnert_{\mathbb{S}(1,0,3,2,0,3)} = \key_{(1,0,4,4,3,6)\setminus(0,0,1,2,3,3)} . \]

Thanks to the stability results for Kohnert polynomials, we have the following.

\begin{corollary}
  Let $a$ be a weak composition. Let $\lambda$ be the partition given by the flattening of the weight of the skew diagram for $a$, i.e.
  \[ \lambda = \mathrm{rev}( \flatten(\wt(\mathbb{S}(a))) ), \]
  and let $\mu$ be the partition given by 
  \[ \mu_i = \lambda_i - \flatten(a)_i . \]
  Then the stable limit of the skew polynomial indexed by $a$ is
  \[ \Kohnert_{\mathbb{S}(a)} = \lim_{m\rightarrow\infty} \kohnert_{\mathbb{S}(a)} = s_{\lambda/\mu} \]
  In particular, skew polynomials are a polynomial generalization of \emph{skew Schur functions}.
  \label{cor:skew}
\end{corollary}

\begin{figure}[ht]
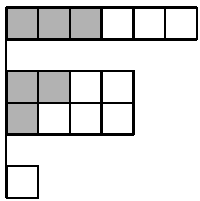

  \begin{displaymath}
    \vline\tableau{ \cb & \cb & \cb & \ & \ & \ \\ \\ \cb & \cb & \ & \ \\ \cb & \ & \ & \ \\ \\ \ }
  \end{displaymath}
  \caption{\label{fig:skew-schur}Illustration of the partitions $\lambda = (6,4,4,1)$ and $\mu = (3,2,1)$ associated to the skew diagram $\mathbb{S}(1,0,3,2,0,3)$.}
\end{figure}

For example, Figure~\ref{fig:skew-schur} illustrates the computation of $\lambda$ and $\mu \subset \lambda$ in establishing the following limit,
\[ \Kohnert_{\mathbb{S}(1,0,3,2,0,3)} = s_{(6,4,4,1)\setminus(3,2,1)} . \]

%%%%%%%%%%%%%%%%%%%%%%%%%%%%%%%%%%%%%%%%%%%%%%%%%%%%%%%%%%%%%%%%
\subsection{Applications of skew polynomials}
%%%%%%%%%%%%%%%%%%%%%%%%%%%%%%%%%%%%%%%%%%%%%%%%%%%%%%%%%%%%%%%%
\label{sec:dem-app}

Unlike Schubert polynomials, whose structure constants enumerate points in a suitable intersection of Schubert varieties and as such as known to be nonnegative, Demazure characters often have negative structure constants. For example,
\[ \key_{(2,0,2)} \key_{(0,2,0)} = \key_{(2, 2, 2)} + \key_{(3, 1, 2)} + \key_{(4, 0, 2)} + \key_{(2, 3, 1)} - \key_{(3, 2, 1)} + \key_{(2, 4, 0)} - \key_{(4, 2, 0)} . \]

Interestingly, the structure constants for skew polynomials are often nonnegative. For example, Figure~\ref{fig:skew-pos} illustrates the following expansion,
\[ \kohnert_{\mathbb{S}(2,0,2)} \cdot \kohnert_{\mathbb{S}(0,2,0)} = \kohnert_{\mathbb{S}(2,2,2)} + \kohnert_{\mathbb{S}(3,1,2)} + \kohnert_{\mathbb{S}(4,0,2)} + \kohnert_{\mathbb{S}(2,3,1)} + \kohnert_{\mathbb{S}(2,4,0)}.\]

\begin{figure}[ht]
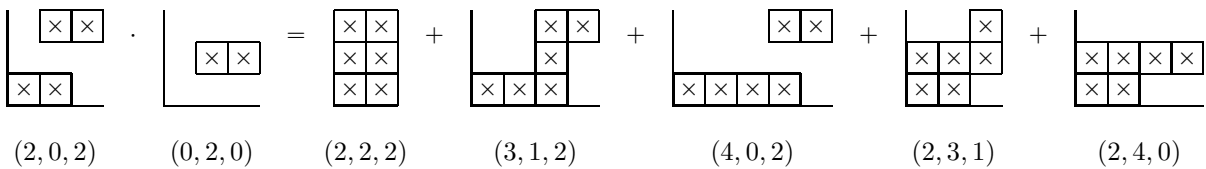

  \begin{displaymath}
    \begin{array}{ccc c ccccc cccc}
      \vline\tableau{ & \times & \times \\ \\ \times & \times \\\hline} & \cdot &
      \vline\tableau{ \\ & \times & \times \\ \\\hline} & = &
      \vline\tableau{ \times & \times \\ \times & \times \\ \times & \times \\\hline } & + & 
      \vline\tableau{ & & \times & \times \\ & & \times \\ \times & \times & \times \\\hline } & + & 
      \vline\tableau{ & & & \times & \times \\ \\ \times & \times & \times & \times \\\hline } & + & 
      \vline\tableau{ & & \times \\ \times & \times & \times \\ \times & \times \\\hline } & + & 
      \vline\tableau{ & \\ \times & \times & \times & \times \\ \times & \times \\\hline } \\ \\
      (2,0,2) & & (0,2,0) & & \makebox[0pt]{$(2,2,2)$} & & (3,1,2) & & (4,0,2) & & (2,3,1) & & (2,4,0)
    \end{array}
  \end{displaymath}
  \caption{\label{fig:skew-pos}Skew diagrams illustrating the skew polynomial expansion of $\kohnert_{\mathbb{S}(2,0,2)} \cdot \kohnert_{\mathbb{S}(0,2,0)}$.}
\end{figure}

In light of Corollary~\ref{cor:skew}, this gives the following nonnegative expansion of a product of skew Schur function into skew Schur functions, 
\[ s_{(3,2)/(1)} \cdot s_{(2)} = s_{(2,2,2)} + s_{(4,3,3)/(2,2)} + s_{(5,4)/(3)} + s_{(3,3,2)/(2)} + s_{(4,2)}. \]
Since skew Schur functions over determine a basis for symmetric functions, this expansion is surprising.

However, such a nice expansion does not always hold. For example, Figure~\ref{fig:skew-neg} illustrates the following signed expansion,
\[ \kohnert_{\mathbb{S}(1,0,1)} \cdot \kohnert_{\mathbb{S}(0,0,1)} = \kohnert_{\mathbb{S}(1,0,2)} + \kohnert_{\mathbb{S}(1,1,1)} + \kohnert_{\mathbb{S}(2,0,1)} - \kohnert_{\mathbb{S}(3,0,0)} . \]

\begin{figure}[ht]
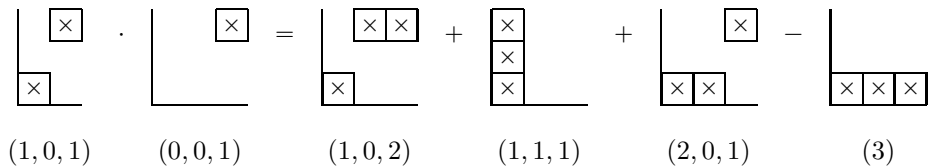

  \begin{displaymath}
    \begin{array}{ccc c cc cc cc c}
      \vline\tableau{ & \times \\ \\ \times & \\\hline} & \cdot &
      \vline\tableau{ & & \times \\ & \\ & \\\hline} & = &
      \vline\tableau{ & \times & \times \\ & \\ \times \\\hline } & + & 
      \vline\tableau{\times & & \\ \times \\ \times \\\hline } & + & 
      \vline\tableau{ & & \times \\ \\ \times & \times \\\hline } & - & 
      \vline\tableau{ \\ \\ \times & \times & \times \\\hline } \\ \\
      (1,0,1) & & (0,0,1) & & (1,0,2) & & (1,1,1) & & (2,0,1) & & (3)
    \end{array}
  \end{displaymath}
  \caption{\label{fig:skew-neg}Skew diagrams illustrating the skew polynomial expansion of $\kohnert_{\mathbb{S}(1,0,1)} \cdot \kohnert_{\mathbb{S}(0,0,1)}$.}
\end{figure}

In this case, Corollary~\ref{cor:skew} still applies and  gives the following signed expansion, 
\[ s_{(2,1)/(1)} \cdot s_{(1)} = s_{(3,1)/(1)} + s_{(1,1,1)} + s_{(3,2)/(2)} - s_{(3)}. \]
Despite the signs, the canonical expansion of a product of skew Schur functions into skew Schur functions is still interesting, and signs appearing can be natural (e.g. see \cite{AM11}). Therefore exploring the structure constants for skew polynomials is a worthwhile endeavor.

Shifting to a more positive direction, skew polynomials correspond with Demazure characters in the case when $a$ is weakly increasing, in which case both are Schur polynomials. Skew polynomials also correspond with Schubert polynomials in certain cases, even outside of the above coincidence with Schur polynomials. For instance, we have the following non-obvious coincidence,
\[ \kohnert_{\mathbb{S}(1,0,3,2,0,3)} = \schubert_{216539478} . \]
Comparing the skew diagram $\mathbb{S}(1,0,3,2,0,3)$ and Rothe diagram $\D(216539478)$ as in Figure~\ref{fig:skew-rothe}, the diagrams themselves are somewhat different yet the resulting Kohnert polynomials coincide.

\begin{figure}[ht]
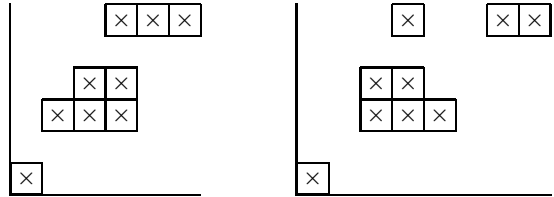

    \begin{displaymath}
      \begin{array}{c@{\hskip 3\cellsize}c}
      \vline\tableau{ & & & \times & \times & \times \\ \\ & & \times & \times \\ & \times & \times & \times \\ \\ \times \\\hline} &
      \vline\tableau{ & & & \times & & & \times & \times \\ \\ & & \times & \times \\ & & \times & \times & \times \\ \\ \times \\\hline}
      \end{array}
    \end{displaymath}
    \caption{\label{fig:skew-rothe}The skew diagram $\mathbb{S}(1,0,3,2,0,3)$ and Rothe diagram $\D(216539478)$.}
\end{figure}

While this coincidence of skew polynomials and Schubert polynomials does not hold in general, we conjecture that Schubert polynomials are, in fact, nested between Demazure characters and skew polynomials in the following sense.

\begin{conjecture}
  Skew polynomials expand as nonnegative sums of Schubert polynomials.
  \label{conj:skew}
\end{conjecture}

Conjecture~\ref{conj:skew} has been verified for degree up to $10$ in up to $6$ variables. If true, this conjecture is highly suggestive that skew polynomials are a combinatorial shadow of representation-theoretic and geometric objects yet to be discovered.

%%%%%%%%%%%%%%%%%%%%%%%%%%%%%%%%%%%%%%%%%%%%%%%%%%%%%%%%%%%%%%%%
%
\section{Extending Schur functions to the ring of quasisymmetric functions}
%
%%%%%%%%%%%%%%%%%%%%%%%%%%%%%%%%%%%%%%%%%%%%%%%%%%%%%%%%%%%%%%%%
\label{sec:right}

We now demonstrate the construction of another Kohnert basis, this one not demazure, with interesting properties. In Section~\ref{sec:right-def} we define the new Kohnert basis of \emph{lock polynomials} and apply our previous results to give explicit formulas for the monomial and fundamental slide expansions. In Section~\ref{sec:right-poly}, we demonstrate how to generate a tableaux model from the corresponding Kohnert diagrams and use this to prove a special case when lock polynomials and Demazure characters coincide. In Section~\ref{sec:right-qsym}, we consider the stable limits of lock polynomials, which we term \emph{extended Schur functions}, since they contain Schur functions and give a basis for quasisymmetric functions.

%%%%%%%%%%%%%%%%%%%%%%%%%%%%%%%%%%%%%%%%%%%%%%%%%%%%%%%%%%%%%%%%
\subsection{Lock polynomials}
%%%%%%%%%%%%%%%%%%%%%%%%%%%%%%%%%%%%%%%%%%%%%%%%%%%%%%%%%%%%%%%%
\label{sec:right-def}

We now define a new Kohnert basis. Using the machinery of Kohnert polynomials, this requires only that for each weak composition $a$, we make a choice for the columns in which we place the $a_i$ boxes in row $i$. For each weak composition $a$, there is a unique right-justified diagram of weight $a$ which we call the \emph{lock diagram for $a$} and denoted by $\DD(a)$. For example, the third diagram in Figure~\ref{fig:diagrams} is the lock diagram for $(0,2,1,2)$. The Kohnert diagrams for this diagram are shown in Figure~\ref{fig:KD-right}.

\begin{figure}[ht]
  \begin{center}
    \begin{displaymath}
      \begin{array}{c@{\hskip \cellsize}c@{\hskip \cellsize}c@{\hskip \cellsize}c@{\hskip \cellsize}c@{\hskip \cellsize}c@{\hskip \cellsize}c@{\hskip \cellsize}c@{\hskip \cellsize}c}
        \vline\tableau{ \times & \times \\ & \times \\ \times & \times \\ & \\\hline} &
        \vline\tableau{ \times & \times \\ & \times \\ \times & \\ & \times \\\hline} &
        \vline\tableau{ \times & \times \\ & \\ \times & \times \\ & \times \\\hline} &
        \vline\tableau{ \times & \\ & \times \\ \times & \times \\ & \times \\\hline} &
        \vline\tableau{ \times & \times \\ & \times \\ & \\ \times & \times \\\hline} &
        \vline\tableau{ & \\ \times & \times \\ \times & \times \\ & \times \\\hline} &
        \vline\tableau{ \times & \times \\ & \\ & \times \\ \times & \times \\\hline} &
        \vline\tableau{ \times & \\ & \times \\ & \times \\ \times & \times \\\hline} &
        \vline\tableau{ & \\ \times & \times \\ & \times \\ \times & \times \\\hline} 
      \end{array}
    \end{displaymath}
    \caption{\label{fig:KD-right}Kohnert diagrams for $\DD(0,2,1,2)$.}
  \end{center}
\end{figure}

\begin{definition}
  The \emph{lock polynomial indexed by $a$} is
  \begin{equation}
    \lock_{a} = \kohnert_{\sDD(a)},
  \end{equation}
  where $\DD(a)$ is the right justified diagram of weight $a$.
  \label{def:right_key}
\end{definition}

For example, from Figure~\ref{fig:KD-right}, we see that
\begin{eqnarray*}
  \lock_{(0,2,1,2)} & = & x_1^2 x_2 x_3^2 + x_1^2 x_2 x_3 x_4 + x_1^2 x_2 x_4^2 + x_1^2 x_3 x_4^2 + x_1 x_2^2 x_3^2 \\
  & & + x_1 x_2^2 x_3 x_4 + x_1 x_2^2 x_4^2 + x_1 x_2 x_3 x_4^2 + x_2^2 x_3 x_4^2 .
\end{eqnarray*}

By Theorem~\ref{thm:kohnert-basis}, since lock polynomials are a Kohnert basis, they are, in particular, a basis of polynomials.

\begin{corollary}
  The lock polynomials form a basis for the polynomial ring that is lower uni-triangular with respect to monomials.
\end{corollary}

By Theorem~\ref{thm:kohnert-mono}, we may express lock polynomials more compactly in the monomial slide basis as follows.

\begin{corollary}
  Lock polynomials expand non-negatively into monomial slide polynomials by
  \begin{equation}
    \lock_a = \sum_{T \in \MKD(\sDD(a))} \mono_{\wt(T)}.
  \end{equation}
\end{corollary}

For the previous example, refined to $\MKD$ in Figure~\ref{fig:MKD}, we have
\begin{eqnarray*}
  \lock_{(0,2,1,2)} & = & \mono_{(0,2,1,2)} + \mono_{(1,1,1,2)} + \mono_{(2,1,1,1)} + \mono_{(1,2,0,2)} + \mono_{(1,2,1,1)}.
\end{eqnarray*}

Even more powerful, by Theorem~\ref{thm:kohnert-slide} we have the following.

\begin{corollary}
  Lock polynomials expand non-negatively into fundamental slide polynomials by
  \begin{equation}
    \lock_{a} = \sum_{T \in \FKD(\sDD(a))} \fund_{\wt(T)}.
  \end{equation}
  \label{cor:lock-fund}
\end{corollary}

\begin{proof}
  Given a lock diagram $\DD(a)$, for every column $c \leq \max(a)$ and for any nonempty collection of rows, there are at least as many cells in those rows in column $c$ as there are in column $c-1$. Therefore \eqref{e:diagram-fund} is always satisfied, and so lock diagrams are fundamental by Definition~\ref{def:diagram-fund}. Thus Theorem~\ref{thm:kohnert-slide} applies.
\end{proof}

Returning again to our example, with the fundamental diagrams shown in Figure~\ref{fig:FKD}, we have
\begin{eqnarray*}
  \lock_{(0,2,1,2)} & = & \fund_{(0,2,1,2)} + \fund_{(1,2,0,2)}.
\end{eqnarray*}

Unlike Schubert polynomials and, trivially, Demazure characters, the lock polynomials do not expand non-negatively into Demazure characters. For example, 
\begin{eqnarray*}
  \lock_{(0,2,1,2)} & = & \key_{(0,2,1,2)} - \key_{(0,2,2,1)} + \key_{(1,2,2,0)} - \key_{(2,2,1,0)} .
\end{eqnarray*}
In light of Conjecture~\ref{conj:demazure}, this comes as no surprise since lock diagrams are not, in general, demazure.

We remark that lock polynomials do not have nonnegative expansions into other familiar bases of the polynomial ring, including quasi-key polynomials \cite{AS-2} and Demazure atoms \cite{LS90}.

%%%%%%%%%%%%%%%%%%%%%%%%%%%%%%%%%%%%%%%%%%%%%%%%%%%%%%%%%%%%%%%%
\subsection{Lock tableaux}
%%%%%%%%%%%%%%%%%%%%%%%%%%%%%%%%%%%%%%%%%%%%%%%%%%%%%%%%%%%%%%%%
\label{sec:right-poly}

Kohnert's rule allows for easy computations, but the potential redundancy of two different sequences of Kohnert moves arriving at the same diagram can be problematic. In \cite{AS-2}, the authors gave a static description of \emph{Kohnert tableaux} for key diagrams by tracking from where each cell in a Kohnert diagram came in the key diagram, and when Kohnert's algorithm gives multiple possibilities for this, fixing a canonical choice. This was done via a canonical labeling of a Kohnert diagram coming from a key diagram, resulting in a simple rule to determine readily if a given diagram can arise as a Kohnert diagram for a key diagram. We extend this procedure to lock diagrams below.

\begin{definition}
  Given a weak composition $a$ of length $n$, a \emph{lock tableau of content $a$} is a diagram filled with entries $1^{a_1}, 2^{a_2}, \ldots, n^{a_n}$, one per cell, satisfying the following conditions:
  \begin{enumerate}[label=(\roman*)]
  \item there is exactly one $i$ in each column from $\max(a)-a_i+1$ through $\max(a)$;
  \item each entry in row $i$ is at least $i$;
  \item the cells with entry $i$ weakly descend from left to right;
  \item the labelling strictly decreases down columns.
  \end{enumerate}
  Denote the set of lock tableaux of content $a$ by $\LT(a)$.
  \label{def:lock-tableaux}
\end{definition}

For example, the lock tableaux of content $(0,2,1,2)$ are shown in Figure~\ref{fig:LT}. Compare this with the Kohnert diagrams for $\DD(0,2,1,2)$ shown in Figure~\ref{fig:KD-right}.

\begin{figure}[ht]
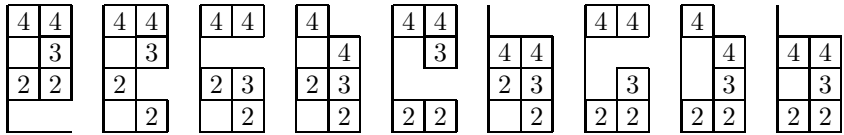

  \begin{center}
    \begin{displaymath}
      \begin{array}{c@{\hskip \cellsize}c@{\hskip \cellsize}c@{\hskip \cellsize}c@{\hskip \cellsize}c@{\hskip \cellsize}c@{\hskip \cellsize}c@{\hskip \cellsize}c@{\hskip \cellsize}c}
        \vline\tableau{ 4 & 4 \\ & 3 \\ 2 & 2 \\ & \\\hline} &
        \vline\tableau{ 4 & 4 \\ & 3 \\ 2 & \\ & 2 \\\hline} &
        \vline\tableau{ 4 & 4 \\ & \\ 2 & 3 \\ & 2 \\\hline} &
        \vline\tableau{ 4 & \\ & 4 \\ 2 & 3 \\ & 2 \\\hline} &
        \vline\tableau{ 4 & 4 \\ & 3 \\ & \\ 2 & 2 \\\hline} &
        \vline\tableau{ & \\ 4 & 4 \\ 2 & 3 \\ & 2 \\\hline} &
        \vline\tableau{ 4 & 4 \\ & \\ & 3 \\ 2 & 2 \\\hline} &
        \vline\tableau{ 4 & \\ & 4 \\ & 3 \\ 2 & 2 \\\hline} &
        \vline\tableau{ & \\ 4 & 4 \\ & 3 \\ 2 & 2 \\\hline} 
      \end{array}
    \end{displaymath}
    \caption{\label{fig:LT}The nine lock tableaux for $\DD(0,2,1,2)$.}
  \end{center}
\end{figure}

The definition of \emph{Kohnert tableaux} in \cite{AS-2}, the analogous model for key diagrams, differs from Definition~\ref{def:lock-tableaux} only in condition (iv). For the Kohnert tableaux case, condition (iv) allowed for an \emph{inversion} in a column, i.e. a pair $i<j$ with $i$ above $j$ in the same column, only if there is an $i$ in the column immediately to the right of and strictly above $j$. Condition (iv) for lock tableaux is far simpler. Moreover, this simplification is forced in the following sense.

\begin{proposition}
  Given a weak composition $a$, let $T$ be any filling of a diagram with entries $1^{a_1}, 2^{a_2}, \ldots, n^{a_n}$, one per cell, satisfying the conditions (i), (ii), and (iii) of Definition~\ref{def:lock-tableaux} and the following
  \begin{itemize}
  \item[(iv)'] if $i<j$ appear in a column with $i$ above $j$, then there is an $i$ in the column immediately to the right of and strictly above $j$.
  \end{itemize}
  Then $T$ is a lock tableau.
  \label{prop:inversion}
\end{proposition}

\begin{proof}
  Suppose $T$ satisfies conditions (i), (ii), and (iii) of Definition~\ref{def:lock-tableaux} and condition (iv)' above. Let $c$ be the rightmost column of $T$ such that there exist entries $i<j$ in column $c$ with $i$ above $j$. By condition (iv)', there is an $i$ in column $c+1$. However, by condition (i), since there is an $i$ in column $c+1$, there must also be a $j$ in column $c+1$. By condition (iii), the $j$ in column $c+1$ lies weakly below the $j$ in column $c$ which, by condition (iv'), lies strictly below the $i$ in column $c+1$, contradicting the choice of $c$. Therefore columns of $T$ must be strictly decreasing top to bottom.
\end{proof}

We mirror the results of \cite{AS-2} for Kohnert tableaux to prove that lock tableaux precisely characterize Kohnert diagrams of lock diagrams.

\begin{lemma}
  For $T \in \LT(a)$, the diagram of $T$ is a Kohnert diagram for $\DD(a)$.
  \label{lem:LT2KD}
\end{lemma}

\begin{proof}
  Fix $T \in \LT(a)$. We claim one can perform reverse Kohnert moves on $T$ to obtain the lock diagram of $a$. Reading the cells of $T$ left to right along rows, starting at the top row, find the first cell, say $C$, whose label is greater than its row index. Any cell above $C$ must have had larger label by condition (iv) of Definition~\ref{def:lock-tableaux}, so by choice of $C$ there cannot be a cell immediately above it, so we may lift $C$ to the row above. To show this reverse move is valid, we need to show there is no cell $C^{\prime}$ to the right of the position in which $C$ lands. Any cell of the landing row must have entry equal to its row index by the choice of $C$, and by condition (ii) $C$ has entry at least that large, so the entry of $C$ is at least as great as the entry of $C^{\prime}$. However, by condition (i), there must be a cell with entry the same as that of $C$ in the column of $C^{\prime}$, and by condition (iii) that cell must be strictly lower. This creates a violation of condition (iv) in the column of $C^{\prime}$ in $T$, a contradiction.

  This procedure clearly preserves condition (i) of Definition~\ref{def:lock-tableaux}. Since this procedure moves the top left cell having a given label greater than its row number, conditions (ii) and (iii) are preserved. Since cells do not change their order within a column, condition (iv) is preserved. Therefore the result is in $\LT(a)$. Iterating this procedure, one eventually obtains the lock diagram with all entries in row $i$ equal to $i$, and each move is a valid reverse Kohnert move. Hence $T$ with its labels removed is in $\KD(\DD(a))$.
\end{proof}

To establish the converse of Lemma~\ref{lem:LT2KD}, we give a canonical labeling of a Kohnert diagram. Once again, the algorithm for Kohnert diagrams coming from a lock diagram is far simpler than the analogous labeling algorithm for Kohnert diagrams coming from a key diagram.

\begin{definition}
  Given $D\in \KD(\DD(a))$, define the \emph{lock labeling of $D$ with respect to $a$}, denoted by $L_a(D)$, by placing the labels $\{i \mid a_i \geq j\}$ to cells of column $\max(a)-j+1$ in increasing order from bottom to top. 
\label{def:labelling_algorithm}
\end{definition}

The lock labeling algorithm is clearly well-defined and establishes the following.

\begin{theorem}
  The labeling map $L_a$ is a weight-preserving bijection between $\KD(\DD(a))$ and $\LT(a)$. In particular, we have
  \begin{equation}
    \lock_a = \sum_{T \in \LT(a)} x^{\wt(T)},
  \end{equation}
  where $\wt(T)$ is the weak composition whose $i$th part is the number of cells in row $i$ of $T$.  
  \label{thm:kohnert}
\end{theorem}	

\begin{proof}
  Suppose that $D\in\KD(\DD(a))$. The lock labeling map $L_a$ is well-defined on $D$ since the number of cells per column is preserved by Kohnert moves. No filling of $D$ other than $L_a(D)$ can give an element of $\LT(a)$ by condition (iv). Therefore, by Lemma~\ref{lem:LT2KD}, removing the labels gives an inverse map provided $L_a(D)$ is a Kohnert tableau.

  Condition (i) of Definition~\ref{def:lock-tableaux} is manifest from the selection of entries, and condition (iv) follows immediately from the lock labeling. For condition (ii), note that every cell in any given column of $D$ must be weakly below where it started, since $D$ is a Kohnert diagram. In particular, for any index $i$ the number of boxes of $D$ appearing below row $i$ is weakly larger than the number of boxes of $\DD(a)$ appearing below row $i$, so since the columns are labeled in increasing order from bottom to top with labels given by the row indices of the original positions of the cells, the label of every cell of $D$ must be weakly larger than its row index.
  
 Clearly condition (iii) holds for the lock labeling of $\DD(a)$, so it is enough to check this condition is preserved under Kohnert moves. Let $D$ be a Kohnert diagram of $D$; assume that (iii) is satisfied under the lock labeling of $D$. When we make a Kohnert move on $D$ and relabel according to the lock labeling, the overall effect is to move a cell $C$ and the interval of cells immediately below $C$ down one space, retaining their labels. Since all labels in the column of $C$ move downwards, clearly this does not introduce any violation of (iii) with entries to the left of the column of $C$. Since we started with a lock diagram, all labels appearing in the interval of cells below $C$ must also appear in the column to the right of $C$, and since we performed a Kohnert move on $C$, the cell immediately right of $C$ in $D$ is empty. Hence $C$'s label appears strictly lower than $C$ in the column right of $C$, and by the lock labeling and the fact that the cells below $C$ form an interval, the same must be true for all cells in the interval below $C$. Hence (iii) is preserved on moving all these cells (with their labels) down one space.
\end{proof}

Recall the quasi-Yamanouchi condition for Kohnert tableaux in \cite{AS-2}.

\begin{definition}
  A lock tableau is \emph{quasi-Yamanouchi} if for each nonempty row $i$, one of the following holds:
  \begin{enumerate}
  \item there is a cell in row $i$ with entry equal to $i$, or
  \item there is a cell in row $i+1$ that lies weakly right of a cell in row $i$.
  \end{enumerate}
  Denote the set of quasi-Yamanouchi lock tableaux of content $a$ by $\QLT(a)$.
  \label{def:quasi-Yam}
\end{definition}

For example, the quasi-Yamanouchi lock tableaux of content $(0,3,2,1)$ are shown in Figure~\ref{fig:QLT}.

\begin{figure}[ht]
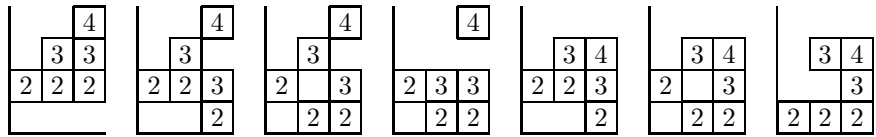

  \begin{center}
    \begin{displaymath}
      \begin{array}{c@{\hskip \cellsize}c@{\hskip \cellsize}c@{\hskip \cellsize}c@{\hskip \cellsize}c@{\hskip \cellsize}c@{\hskip \cellsize}c}
        \vline\tableau{ & & 4 \\ & 3 & 3 \\ 2 & 2 & 2 \\ & & \\\hline} &
        \vline\tableau{ & & 4 \\ & 3 & \\ 2 & 2 & 3 \\ & & 2 \\\hline} &
        \vline\tableau{ & & 4 \\ & 3 & \\ 2 & & 3 \\ & 2 & 2 \\\hline} &
        \vline\tableau{ & & 4 \\ & & \\ 2 & 3 & 3 \\ & 2 & 2 \\\hline} &
        \vline\tableau{ & & \\ & 3 & 4 \\ 2 & 2 & 3 \\ & & 2 \\\hline} &
        \vline\tableau{ & & \\ & 3 & 4 \\ 2 & & 3 \\ & 2 & 2 \\\hline} &
        \vline\tableau{ & & \\ & 3 & 4 \\ & & 3 \\ 2 & 2 & 2 \\\hline} 
      \end{array}
    \end{displaymath}
    \caption{\label{fig:QLT}The set $\QLT(0,3,2,1)$ of quasi-Yamanouchi lock tableaux of content $(0,3,2,1)$ .}
  \end{center}
\end{figure}

Quasi-Yamanouchi lock tableaux allow for the following re-characterization of the fundamental slide expansion of lock polynomials.

\begin{theorem}
  Lock polynomials expand non-negatively into fundamental slide polynomials by
  \begin{equation}
    \lock_{a} = \sum_{T \in \QLT(a)} \fund_{\wt(T)}.
  \end{equation}
  \label{thm:lock-fund-2}
\end{theorem}

\begin{proof}
  We define a \emph{de-standardization map} from $\LT(a)$ to $\QLT(a)$ by sending a lock tableau $T$ to the quasi-Yamanouchi lock tableau $\destand(T)$ constructed as follows. For each row, say $i$, if every cell in row $i$ lies strictly right of every cell in row $i+1$ and the leftmost cell of row $i$ has label larger than $i$, then move every cell in row $i$ up to row $i+1$. Repeat until no such row exists. To see that the de-standardization map maintains the lock tableau conditions, note that the labels within each column are maintained, proving (i). De-standardization does not move cells to a row higher than their label,so (ii) is maintained. No cell is moved from weakly below to strictly above any other, and no cell moves upward if there is a cell to its right in the row above, so conditions (iii) and (iv) are maintained. Finally, by definition de-standardization terminates if and only if the quasi-Yamanouchi condition is met.

  Let $T\in \LT(a)$ and suppose $\destand(T) = S \in \QLT(a)$. Since $\destand$ moves cells upwards we have $\wt(T)\ge \wt(S)$, and since $\destand$ moves\emph{all} cells in row $i$ to row $i+1$, we have $\flatten(\wt(T))$ refines $\flatten(\wt(S))$. Hence $x^T$ is a monomial of $\fund_{\wt(S)}$. Conversely, let $S\in \QLT(a)$, and let $b$ be a weak composition such that $b\ge \wt(S)$ and $\flatten(b)$ refines $\flatten(\wt(S))$. We show there is a unique $T\in \LT(a)$ with $\wt(T) = b$ and $\destand(T)=S$.  To reconstruct $T$ from $b$ and $U$, for $j = 1,\ldots,n$, if $\wt(S)_{j} = b_{i_{j-1} + 1} + \cdots + b_{i_{j}}$, then, from right to left, move the first $b_{i_{j-1} + 1}$ cells down to row $i_{j-1} + 1$, the next $b_{i_{j-1} + 2}$ cells down to row $i_{j-1} + 2$, and so on. By construction $T$ is a Kohnert diagram of $S$ (so $T\in \LT(a)$), $\wt(T) =b$, and $\destand(T)=S$. Uniqueness follows from the lack of choice at every step.
\end{proof}

For example, from Figure~\ref{fig:QLT} we may quickly compute
\begin{displaymath}
  \lock_{(0,3,2,1)} = \fund_{(0,3,2,1)} + \fund_{(1,3,1,1)} + \fund_{(2,2,1,1)} + \fund_{(2,3,0,1)} + \fund_{(1,3,2,0)} + \fund_{(2,2,2,0)} + \fund_{(3,1,2,0)}.
\end{displaymath}

While we noted that lock polynomials are not, in general, nonnegative on Demazure characters, there is a case where lock polynomials and Demazure characters coincide. Letting Conjecture~\ref{conj:demazure} be our guide, we notice that a lock diagram $\DD(a)$ is demazure if and only if $\flatten(a)$ is weakly decreasing. Thus the following result lends more weight to Conjecture~\ref{conj:demazure}.

\begin{theorem}
  Given a weak composition $a$ such that $\flatten(a)$ is weakly decreasing, we have
  \begin{equation}
    \lock_{a} = \key_{a}.
    \label{e:LR}
  \end{equation}
  \label{thm:LR}
\end{theorem}

\begin{proof}
  We utilize the machinery of weak dual equivalence \cite{Ass-W} to establish that $\lock_a$ expands nonnegatively into Demazure characters when $\flatten(a)$ is weakly decreasing. For $a$ a weak composition of $n$, we must define involutions $\psi_2,\ldots,\psi_{n-1}$ on $\QLT(a)$ such that $\psi_i \psi_j (T) = \psi_j \psi_i (T)$ whenever $|i-j| \geq 3$ and for $i-h \leq 3$, there exists a weak composition $b$ of $i-h+3$ such that
  \[ \sum_{U \in [T]_{(h,i)}} \fund_{\wt_{(h-1,i+1)}(U)} = \key_{b}, \]
  where $[T]_{(h,i)}$ is the equivalence class generated by $\psi_h,\ldots,\psi_i$, and $\wt_{(h,i)}(U)$ is the weak composition of $i-h+1$ obtained by deleting the first $h-1$ and last $n-i$ nonzero parts from $\wt(U)$.

  To define the desired involutions, we first relabel the cells of $T\in\QLT(a)$ with $1,2,\ldots,n$ from the bottom row up, labeling each row right to left, then raise the cells to $\DD(a)$ maintaining their relative order. The result is a bijective filling of $\DD(a)$ with $1,2,\ldots,n$ such that rows decrease left to right and columns decrease top to bottom. This process is reversible by lowering cells of a bijective filling of $\DD(a)$ such that $1$ is the lowest, $2$ the next lowest, and so on, and then applying the de-standardization map from the proof of Theorem~\ref{thm:lock-fund-2}. In fact, if we allow cells to fall below the $x$-axis, then this establishes a bijection between $\QLT(a)$ and bijective fillings of $\DD(a)$ decreasing rows and columns. 

  For $2 \leq i \leq n-1$, let $\psi_i$ act on $T \in \QLT(a)$ by instead acting on bijective fillings of $\DD(a)$ with decreasing rows and columns as follows. If, in reading entries right to left from the top row down, $i\pm 1$ lies between $i$ and $i \mp 1$, then $\psi_i$ exchanges $i$ and $i \mp 1$; otherwise $\psi_i$ acts by the identity. To see that this is well-defined, there are two cases to check. If $i+1$ lies above $i$ in the same column, then $i-1$ lies between them in the previous sense only if it lies right of $i$. Since $\DD(a)$ is right justified, this forces an entry $j$ right of $i+1$ and above $i-1$. The decreasing rows and columns forces $j=i$, which is a contradiction. Similarly, if $i$ lies above $i-1$ in the same column, then $i+1$ lies between them in the previous sense only if it lies left of $i$. If $\flatten(a)$ is weakly decreasing, then this forces an entry $j$ below $i+1$ and left of $i-1$. The decreasing rows and columns forces $j=i$, which is again a contradiction. Therefore when $\flatten(a)$ is weakly decreasing, $\psi_i$ is well-defined. 

  Given the local nature of $\psi_i$, since $\{i-1,i,i+1\} \cap \{j-1,j,j+1\} = \varnothing$ whenever $|i-j| \geq 3$, we clearly have the commutativity relation $\psi_i \psi_j (T) = \psi_j \psi_i (T)$. The second condition is local, requiring between three and six consecutively labeled cells that must fit inside a staircase diagram. Therefore there are finitely many cases to check, which can be verified easily by direct (and tedious) enumeration or by computer. We have done both, so Demazure positivity follows. To see that this is a single Demazure character, we note that bijective fillings of $\DD(a)$ with decreasing rows and columns are in bijection with bijective fillings of partition shape $\flatten(a)$ with increasing rows and columns, the latter of which are standard Young tableaux that generate a Schur function. Therefore in the stable limit we have a single term in the Schur expansion, so the non-negativity together with Proposition~\ref{prop:key-stable} implies the Demazure expansion must have a single term as well. By the unique leading term for lock polynomials and Demazure characters, we must have $\lock_a = \key_a$.
\end{proof}

%%%%%%%%%%%%%%%%%%%%%%%%%%%%%%%%%%%%%%%%%%%%%%%%%%%%%%%%%%%%%%%%
\subsection{The extended Schur basis}
%%%%%%%%%%%%%%%%%%%%%%%%%%%%%%%%%%%%%%%%%%%%%%%%%%%%%%%%%%%%%%%%
\label{sec:right-qsym}

By Theorem~\ref{thm:kohnert-stable}, we may consider the stable limits of lock polynomials, which we call \emph{extended Schur functions}.

\begin{definition}
  Given a (strong) composition $\alpha$, the \emph{extended Schur polynomial indexed by $\alpha$} is given by
  \begin{equation}
    \eS_{\alpha}(x_1,\ldots,x_m) = \lock_{0^m \times \alpha}(x_1,\ldots,x_m,0,\ldots,0),
    \label{e:extended-poly}
  \end{equation}
  and the \emph{extended Schur function indexed by $\alpha$} is given by
  \begin{equation}
    \eS_{\alpha}(X) = \lim_{m \rightarrow \infty} \lock_{0^m \times \alpha} = \Kohnert_{\sDD(\alpha)}.
    \label{e:extended-function}
  \end{equation}  
  \label{def:eSchur}
\end{definition}

For example, we can compute the extended Schur function $\eS_{(2,1,2)}(X)$ by
\begin{eqnarray*}
  \lock_{(2,1,2)} & = & \fund_{(2,1,2)}, \\
  \lock_{(0,2,1,2)} & = & \fund_{(0,2,1,2)} + \fund_{(1,2,0,2)}, \\
  \lock_{(0,0,2,1,2)} & = & \fund_{(0,0,2,1,2)} + \fund_{(0,1,2,0,2)} + \fund_{(1,1,2,0,1)}, \\
  & \vdots & \\
  \eS_{(2,1,2)} & = & F_{(2,1,2)} + F_{(1,2,2)} + F_{(1,1,2,1)}.
\end{eqnarray*}

By Proposition~\ref{prop:delete-row}, we have the following statement showing that the extended Schur functions include all Kohnert quasisymmetric functions for lock diagrams.

\begin{corollary}
  Given a weak composition $a$, we have
  \[ \Kohnert_{\sDD(a)} = \eS_{\flatten(a)}. \]
\end{corollary}

To justify the name \emph{extended Schur functions}, we have the following result.

\begin{proposition}
  For $\lambda$ a partition, we have
  \begin{equation}
    \eS_{\lambda}(X) = s_{\lambda}(X).
  \end{equation}
  \label{prop:eSchur}
\end{proposition}

\begin{proof}
  By Theorem~\ref{thm:LR}, we have $\lock_{0^m \times \lambda} = \key_{0^m \times \lambda}$ since $\lambda$ is weakly decreasing. By Proposition~\ref{prop:key-stable}, we have $\lim_{m\rightarrow\infty} \key_{0^m \times \lambda} = s_{\lambda}(X)$. The result now follows from Definition~\ref{def:eSchur}.
\end{proof}

Note that Demazure characters do not expand non-negatively into lock polynomials. For example,
\begin{eqnarray*}
  \key_{(0,2,1,2)} & = & \lock_{(0,2,1,2)} + \lock_{(0,2,2,1)} - \lock_{(1,2,2,0)}.
\end{eqnarray*}
Conversely, Proposition~\ref{prop:eSchur} shows that the stable limits of Demazure characters do expand nonnegatively into the stable limits of lock polynomials. 

We can use lock tableaux to give a tableaux model for extended Schur functions as follows.

\begin{definition}
  Given a (strong) composition $\alpha$, a \emph{semi-standard extended tableau of shape $\alpha$} is a filling of $\DD(\alpha)$ with positive integers such that rows weakly decrease left to right and columns strictly decrease top to bottom. Denote the set of semi-standard extended tableaux of shape $\alpha$ by $\SSET(\alpha)$, or by $\SSET_n(\alpha)$ if we restrict the integers to $\{1,2,\ldots,n\}$.
  \label{def:SSET}
\end{definition}

For example, the semi-standard extended tableaux of shape $(2,1,2)$ are shown in Figure~\ref{fig:SSET}. Compare these with the lock tableaux for $(0,2,1,2)$ shown in Figure~\ref{fig:LT}.

\begin{figure}[ht]
  \begin{center}
    \begin{displaymath}
      \begin{array}{c@{\hskip \cellsize}c@{\hskip \cellsize}c@{\hskip \cellsize}c@{\hskip \cellsize}c@{\hskip \cellsize}c@{\hskip \cellsize}c@{\hskip \cellsize}c@{\hskip \cellsize}c}
        \tableau{4 & 4 \\ & 3 \\ 2 & 2 } &
        \tableau{4 & 4 \\ & 3 \\ 2 & 1 } &
        \tableau{4 & 4 \\ & 3 \\ 1 & 1 } &
        \tableau{4 & 4 \\ & 2 \\ 2 & 1 } &
        \tableau{4 & 4 \\ & 2 \\ 1 & 1 } &
        \tableau{4 & 3 \\ & 2 \\ 2 & 1 } &
        \tableau{4 & 3 \\ & 2 \\ 1 & 1 } &
        \tableau{3 & 3 \\ & 2 \\ 2 & 1 } &
        \tableau{3 & 3 \\ & 2 \\ 1 & 1 } 
      \end{array}
    \end{displaymath}
    \caption{\label{fig:SSET}The set $\SSET_4(2,1,2)$ of semi-standard extended tableaux of shape $(2,1,2)$ with entries in $\{1,2,3,4\}$.}
  \end{center}
\end{figure}

For $T$ a semi-standard extended tableau, let $\wt(T)$ be the weak composition whose $i$th part is the number of entries of $T$ equal to $i$. Then extended Schur functions are the generating function for extended tableaux.

\begin{theorem}
  For $\alpha$ a (strong) composition, we have
  \begin{eqnarray*}
    \eS_{\alpha}(x_1,\ldots,x_n) & = & \sum_{T \in \SSET_n(\alpha)} x_1^{\wt(T)_1} \cdots x_n^{\wt(T)_n}, \\
    \eS_{\alpha}(X) & = & \sum_{T \in \SSET(\alpha)} X^{\wt(T)},
  \end{eqnarray*}
  where $X^{\wt(T)}$ is the monomial $x_1^{\wt(T)_1} \cdots x_n^{\wt(T)_n}$ when $\wt(T)$ has length $n$.
  \label{thm:SSET}
\end{theorem}

\begin{proof}
  By Theorem~\ref{thm:kohnert} and Definition~\ref{def:eSchur}, we have that $\eS_{\alpha}(x_1,\ldots,x_m)$ is the generating polynomial for lock tableaux of content $(0^m \times \alpha)$ with no cells above row $m$. Given such a lock tableau $D$, we may define a semi-standard extended tableau $T$ as follows. For every cell $x$ of $D$, place an entry equal to the row index of $x$ into the cell of $T$ in the same column as $x$ and in the row given by $m$ minus the entry of $x$. For example, the lock tableaux in Figure~\ref{fig:LT} map to the semi-standard extended tableau in Figure~\ref{fig:SSET}, respectively. To reverse the procedure, given a semi-standard extended tableau $T$, we may construct a lock tableau $D$ by raising $T$ up $m$ rows then moving each cell of $T$ down to the row equal to its entry. 

  Definition~\ref{def:lock-tableaux} condition (i) is equivalent to $T$ having shape $\DD(a)$, condition (ii) and the restriction to lock tableaux with all entries weakly below row $m$ is equivalent to $T$ having labels in $\{1,2,\ldots,m\}$, condition (iii) is equivalent to rows of $T$ weakly decreasing, and condition (iv) is equivalent to columns of $T$ strictly decreasing. Moreover, $\wt(D) = \wt(T)$. Therefore this gives a weight-preserving bijection between $\LT(0^m \times \alpha)$ with all cells weakly below row $m$ and $\SSET_m(\alpha)$, so the first formula follows. The second follows from the first by letting $m$ go to infinity.
\end{proof}

Campbell, Feldman, Light, Shuldiner and Xu \cite{CFLSX14} defined the same class of tableaux of composition shape, which they termed \emph{shin-tableaux}, when they introduced the \emph{shin functions}, which are a basis of symmetric functions in non-commuting variables that generalize the Schur functions. As these tableaux characterize the expansion of noncommutative homogeneous symmetric functions into shin functions, they also characterize the expansion of the dual basis, which are quasisymmetric functions, into monomial quasisymmetric functions. In other words, the extended Schur functions are the dual basis to the noncommutative shin functions. In \cite{CFLSX14}, the authors observe this and state the positive expansion into monomial quasisymmetric functions. Below we develop further properties of this basis.

Just as the fundamental slide expansion of lock polynomials is a more compact formula, we may translate Theorem~\ref{thm:SSET} into a fundamental quasisymmetric function expansion using the following.

\begin{definition}
  A \emph{standard extended tableau of shape $\alpha$} is a semi-standard extended tableau of shape $\alpha$ that uses each of the integers $1,2,\ldots,n$ exactly once. Denote the set of standard extended tableaux of shape $\alpha$ by $\SET(\alpha)$, and call the element of $\SET(\alpha)$ whose entries in row $i+1$ are the first $\alpha_{i+1}$ integers larger than $\alpha_1 + \ldots + \alpha_i$ the \emph{super-standard} extended tableau of shape $\alpha$.
  \label{def:SET}
\end{definition}

For example, the standard extended tableaux of shape $(2,1,2)$ are shown in Figure~\ref{fig:SET}. The leftmost is the super-standard one.

\begin{figure}[ht]
  \begin{center}
    \begin{displaymath}
      \begin{array}{c@{\hskip 2\cellsize}c@{\hskip 2\cellsize}c}
        \tableau{5 & 4 \\ & 3 \\ 2 & 1 } &
        \tableau{5 & 4 \\ & 2 \\ 3 & 1 } &
        \tableau{5 & 3 \\ & 2 \\ 4 & 1 } 
      \end{array}
    \end{displaymath}
    \caption{\label{fig:SET}The set $\SET(2,1,2)$ of standard extended tableaux of shape $(2,1,2)$.}
  \end{center}
\end{figure}

For $T$ a standard extended tableau, define the \emph{descent composition of $T$}, denoted by $\Des(T)$, to be the (strong) composition given by increasing runs of the entries $1,2,\ldots,n$ when read right to left in $T$. For example, the descent compositions for the standard extended tableaux in Figure~\ref{fig:SET} are $(2,1,2)$, $(1,2,2)$, and $(1,1,2,1)$, respectively; note the descent composition of the super-standard extended tableau is $\alpha$. Compare this with the $F$-expansion of $\eS_{(2,1,2)}$.

\begin{theorem}
  For $\alpha$ a (strong) composition, we have
  \begin{equation}
    \eS_{\alpha}(X) = \sum_{T \in \SET(\alpha)} F_{\Des(T)} (X).
  \end{equation}
  \label{thm:SET}
\end{theorem}

\begin{proof}
  Similar to the proof of Theorem~\ref{thm:lock-fund-2}, we construct a \emph{standardization map} from $\SSET(\alpha)$ to $\SET(\alpha)$ by reading entries from smallest to largest and reading cells with entry $i$ from right to left, change the entries to $1,2,3,\ldots,n$. It is easy to see that this maintains the extended tableau conditions, so the result is a standard extended tableau.

  Let $T\in \SSET(\alpha)$ and suppose that $T$ standardizes to $S \in \SET(\alpha)$. By construction and the definition of the descent composition, we have $\flatten(\wt(T))$ refines $\Des(S)$. Hence $x^T$ is a monomial of $F_{\Des(S)}$. Conversely, let $S\in \SET(\alpha)$, and let $b$ be a weak composition such that $\flatten(b)$ refines $\Des(S)$. We show there is a unique $T\in SSET(\alpha)$ with $\wt(T) = b$ that standardizes to $S$. To reconstruct $T$ from $b$ and $S$, for $j = 1,\ldots,n$, if $\Des(S)_{j} = b_{i_{j-1} + 1} + \cdots + b_{i_{j}}$, then, from right to left, set the first $b_{i_{j-1} + 1}$ cells to have entry $i_{j-1} + 1$, the next $b_{i_{j-1} + 2}$ cells to have entry $i_{j-1} + 2$, and so on. This maintains the extended tableaux conditions, $\wt(T) =b$, and $T$ standardizes to $S$. Uniqueness follows from the lack of choice at every step.
\end{proof}

Using this characterization, we now justify the terminology extended Schur \emph{basis}.

\begin{theorem}
  The extended Schur functions form a basis for the ring of quasisymmetric functions.
\end{theorem}

\begin{proof}
  Clearly, the descent composition of the super-standard extended tableau of shape $\alpha$ is larger in lexicographic order than the descent composition of any other element of $\SET(\alpha)$. Hence by Theorem~\ref{thm:SET}, the extended Schur functions are upper uni-triangular with respect to lexicographic order on the fundamental quasisymmetric functions.
\end{proof}

The extended Schur basis exhibits many nice properties and should have interesting applications to symmetric and quasisymmetric functions. We close our introduction of this basis with two such properties.

\begin{proposition}
  Let $\alpha$ be a (strong) composition and let $\beta$ be obtained from $\alpha$ by exchanging two adjacent parts $\alpha_i < \alpha_{i+1}$. Then the difference $\eS_{\beta} - \eS_{\alpha}$ is $F$-positive. In particular, the terms of the fundamental quasisymmetric expansion of $\eS_{\alpha}$ are a sub(multi)set of the terms of $s_{\lambda}$ where $\lambda = \sort(\alpha)$.
  \label{prop:swap}
\end{proposition}

\begin{proof}
  Define a map from $\SET(\alpha)$ to $\SET(\beta)$ by dropping the leftmost $\alpha_{i+1}-\alpha_{i}$ cells of row $i+1$ of $\DD(\alpha)$ down one row, retaining all entries. This map is well-defined since all cells retain their relative order within columns, and since in any element of $\SET(\alpha)$ the leftmost cell of row $i$ has smaller entry than the cell immediately above it, which in turn has smaller entry than the cell immediately left of it (which is the rightmost cell to drop down). The map is clearly injective, and moreover preserves $\Des$ since we only move entries within their original column.
\end{proof}

For example, taking $\alpha = (2,1,2)$ and exchanging $\alpha_2$ and $\alpha_3$ to get $\beta = (2,2,1)$, we have
\[ \eS_{(2,2,1)} - \eS_{(2,1,2)} = F_{(2,2,1)} + F_{(1,2,1,1)}. \]
For further examples, compare entries for the extended Schur functions in Table~\ref{tab:eSchur}.

\begin{proposition}
The extended Schur function $\eS_\alpha$ is equal to a single fundamental quasisymmetric function $\fund_\alpha$ if and only if $\DD(\alpha)$ is a (reverse) hook shape, i.e. $\alpha = (1^k, \ell)$ for some $k$ and $\ell$.
\end{proposition}

\begin{proof}
If $\alpha$ is a reverse hook shape, there is clearly only one standard extended tableau of shape $\alpha$, specifically the super-standard one. Conversely, suppose $\alpha_i > 1$ and $\alpha_{i+1} > 0$. Then a second element of $\SET(\alpha)$ may be obtained from the super-standard one by swapping the entry of the leftmost cell of row $i$ and the rightmost cell of row $i+1$.
\end{proof}

\begin{table}[ht]
  \begin{displaymath}
    \begin{array}{rcl}
    \eS_{(1)}   & = & F_{(1)} \\[1ex]
    \eS_{(2)}   & = & F_{(2)} \\
    \eS_{(11)}  & = & F_{(11)} \\[1ex]
    \eS_{(3)}   & = & F_{(3)} \\
    \eS_{(21)}  & = & F_{(12)} + F_{(21)} \\
    \eS_{(12)}  & = & F_{(12)} \\
    \eS_{(111)} & = & F_{(111)} \\[1ex]
    \eS_{(4)}   & = & F_{(4)} \\
    \eS_{(31)}  & = & F_{(13)} + F_{(22)} + F_{(31)} \\
    \eS_{(13)}  & = & F_{(13)} \\
    \eS_{(22)}  & = & F_{(121)} + F_{(22)} \\
    \eS_{(211)} & = & F_{(112)} + F_{(121)} + F_{(211)}  \\
    \eS_{(121)} & = & F_{(112)} + F_{(121)} \\
    \eS_{(112)} & = & F_{(112)} \\
    \eS_{(1111)}& = & F_{(1111)} 
    \end{array} \hspace{1em}
    \begin{array}{rcl}
    \eS_{(5)}    & = & F_{(5)} \\
    \eS_{(41)}   & = & F_{(14)} + F_{(23)} + F_{(32)} + F_{(41)} \\
    \eS_{(14)}   & = & F_{(14)} \\
    \eS_{(32)}   & = & F_{(23)} + F_{(122)} + F_{(131)} + F_{(221)} + F_{(32)} \\
    \eS_{(23)}   & = & F_{(23)} + F_{(122)} \\
    \eS_{(311)}  & = & F_{(113)} + F_{(122)} + F_{(131)} + F_{(212)} + F_{(221)} + F_{(311)} \\
    \eS_{(131)}  & = & F_{(113)} + F_{(122)} + F_{(131)} \\
    \eS_{(113)}  & = & F_{(113)} \\
    \eS_{(221)}  & = & F_{(122)} + F_{(1121)} + F_{(212)} + F_{(1211)} + F_{(221)} \\
    \eS_{(212)}  & = & F_{(122)} + F_{(1121)} + F_{(212)} \\
    \eS_{(122)}  & = & F_{(122)} + F_{(1121)} \\
    \eS_{(2111)} & = & F_{(1112)} + F_{(1121)} + F_{(1211)} + F_{(2111)} \\
    \eS_{(1211)} & = & F_{(1112)} + F_{(1121)} + F_{(1211)} \\
    \eS_{(1121)} & = & F_{(1112)} + F_{(1121)} \\
    \eS_{(1112)} & = & F_{(1112)} \\
    \eS_{(11111)}& = & F_{(11111)} 
    \end{array}
  \end{displaymath}
  \caption{\label{tab:eSchur}A table of the fundamental expansion of the extended Schur functions.\vspace{-1\baselineskip}}
\end{table}

%%%%%%%%%%%%%%%%%%%%%%%%%%%%%%%%%%%%%%%%%%%%%%%%%%%%%%%%%%%%
%
%  Bibliography
%
%%%%%%%%%%%%%%%%%%%%%%%%%%%%%%%%%%%%%%%%%%%%%%%%%%%%%%%%%%%%

\bibliographystyle{amsalpha} 
\bibliography{kohnert}

\newcommand{\etalchar}[1]{$^{#1}$}
\providecommand{\bysame}{\leavevmode\hbox to3em{\hrulefill}\thinspace}
\providecommand{\MR}{\relax\ifhmode\unskip\space\fi MR }
\providecommand{\MRhref}[2]{%
  \href{http://www.ams.org/mathscinet-getitem?mr=#1}{#2}
}
\providecommand{\href}[2]{#2}
\begin{thebibliography}{CFL{\etalchar{+}}14}

\bibitem[AM11]{AM11}
Sami~H. Assaf and Peter R.~W. McNamara, \emph{A {P}ieri rule for skew shapes},
  J. Combin. Theory Ser. A \textbf{118} (2011), no.~1, 277--290.

\bibitem[AS17]{AS17}
Sami Assaf and Dominic Searles, \emph{Schubert polynomials, slide polynomials,
  {S}tanley symmetric functions and quasi-{Y}amanouchi pipe dreams}, Adv. Math.
  \textbf{306} (2017), 89--122.

\bibitem[AS18]{AS-2}
\bysame, \emph{{K}ohnert tableaux and a lifting of quasi-{S}chur functions}, J.
  Combin. Theory Ser. A \textbf{156} (2018), 85--118.

\bibitem[Ass17a]{Ass-R}
Sami Assaf, \emph{Combinatorial models for {S}chubert polynomials},
  arXiv:1703.00088, 2017.

\bibitem[Ass17b]{Ass-W}
\bysame, \emph{Weak dual equivalence for polynomials}, arXiv:1702.04051, 2017.

\bibitem[BB93]{BB93}
Nantel Bergeron and Sara Billey, \emph{R{C}-graphs and {S}chubert polynomials},
  Experiment. Math. \textbf{2} (1993), no.~4, 257--269.

\bibitem[BJS93]{BJS93}
Sara~C. Billey, William Jockusch, and Richard~P. Stanley, \emph{Some
  combinatorial properties of {S}chubert polynomials}, J. Algebraic Combin.
  \textbf{2} (1993), no.~4, 345--374.

\bibitem[CFL{\etalchar{+}}14]{CFLSX14}
John Campbell, Karen Feldman, Jennifer Light, Pavel Shuldiner, and Yan Xu,
  \emph{A {S}chur-like basis of {NS}ym defined by a {P}ieri rule}, Electron. J.
  Combin. \textbf{21} (2014), no.~3, Paper 3.41, 19.

\bibitem[Dem74a]{Dem74a}
Michel Demazure, \emph{D\'esingularisation des vari\'et\'es de {S}chubert
  g\'en\'eralis\'ees}, Ann. Sci. \'Ecole Norm. Sup. (4) \textbf{7} (1974),
  53--88, Collection of articles dedicated to Henri Cartan on the occasion of
  his 70th birthday, I.

\bibitem[Dem74b]{Dem74}
\bysame, \emph{Une nouvelle formule des caract\`eres}, Bull. Sci. Math. (2)
  \textbf{98} (1974), no.~3, 163--172.

\bibitem[EG87]{EG87}
Paul Edelman and Curtis Greene, \emph{Balanced tableaux}, Adv. in Math.
  \textbf{63} (1987), no.~1, 42--99.

\bibitem[EML53]{EM53}
Samuel Eilenberg and Saunders Mac~Lane, \emph{On the groups of {$H(\Pi,n)$}.
  {I}}, Ann. of Math. (2) \textbf{58} (1953), 55--106.

\bibitem[Ful92]{Ful92}
William Fulton, \emph{Flags, {S}chubert polynomials, degeneracy loci, and
  determinantal formulas}, Duke Math. J. \textbf{65} (1992), no.~3, 381--420.

\bibitem[Ges84]{Ges84}
Ira~M. Gessel, \emph{Multipartite {$P$}-partitions and inner products of skew
  {S}chur functions}, Combinatorics and algebra (Boulder, Colo., 1983),
  Contemp. Math., vol.~34, Amer. Math. Soc., Providence, RI, 1984,
  pp.~289--317.

\bibitem[Hof00]{Hof00}
Michael~E. Hoffman, \emph{Quasi-shuffle products}, J. Algebraic Combin.
  \textbf{11} (2000), no.~1, 49--68.

\bibitem[Koh91]{Koh91}
Axel Kohnert, \emph{Weintrauben, {P}olynome, {T}ableaux}, Bayreuth. Math. Schr.
  (1991), no.~38, 1--97, Dissertation, Universit{\"a}t Bayreuth, Bayreuth,
  1990.

\bibitem[KP87]{KP1}
Witold Kra\'skiewicz and Piotr Pragacz, \emph{Foncteurs de schubert}, C. R.
  Acad. Sci. Paris S\'er. I. Math. (1987), no.~304, 209--211.

\bibitem[KP04]{KP2}
\bysame, \emph{Schubert functors and schubert polynomials}, European J. Combin.
  (2004), no.~25, 1327--1344.

\bibitem[LS82]{LS82}
Alain Lascoux and Marcel-Paul Sch{\"u}tzenberger, \emph{Polyn\^omes de
  {S}chubert}, C. R. Acad. Sci. Paris S\'er. I Math. \textbf{294} (1982),
  no.~13, 447--450.

\bibitem[LS85]{LS85}
Alain Lascoux and Marcel-Paul Sch\"utzenberger, \emph{Schubert polynomials and
  the {L}ittlewood-{R}ichardson rule}, Lett. Math. Phys. \textbf{10} (1985),
  no.~2-3, 111--124.

\bibitem[LS90]{LS90}
Alain Lascoux and Marcel-Paul Sch{\"u}tzenberger, \emph{Keys \& standard
  bases}, Invariant theory and tableaux ({M}inneapolis, {MN}, 1988), IMA Vol.
  Math. Appl., vol.~19, Springer, New York, 1990, pp.~125--144.

\bibitem[Mac91]{Mac91}
I.~G. Macdonald, \emph{Notes on {S}chubert polynomials}, LACIM, Univ. Quebec a
  Montreal, Montreal, PQ, 1991.

\bibitem[RS95a]{RS95}
Victor Reiner and Mark Shimozono, \emph{Key polynomials and a flagged
  {L}ittlewood-{R}ichardson rule}, J. Combin. Theory Ser. A \textbf{70} (1995),
  no.~1, 107--143.

\bibitem[RS95b]{RS95-2}
\bysame, \emph{Specht series for column-convex diagrams}, J. Algebra
  \textbf{174} (1995), no.~2, 489--522.

\bibitem[RS98]{RS98}
\bysame, \emph{Percentage-avoiding, northwest shapes and peelable tableaux}, J.
  Combin. Theory Ser. A \textbf{82} (1998), no.~1, 1--73.

\bibitem[Sta84]{Sta84}
Richard~P. Stanley, \emph{On the number of reduced decompositions of elements
  of {C}oxeter groups}, European J. Combin. \textbf{5} (1984), no.~4, 359--372.

\bibitem[Win99]{Win99}
Rudolf Winkel, \emph{Diagram rules for the generation of {S}chubert
  polynomials}, J. Combin. Theory Ser. A \textbf{86} (1999), no.~1, 14--48.

\bibitem[Win02]{Win02}
\bysame, \emph{A derivation of {K}ohnert's algorithm from {M}onk's rule},
  S\'em. Lothar. Combin. \textbf{48} (2002), Art.\ B48f, 14.

\end{thebibliography}

\end{document}